\documentclass{amsart}
\usepackage{amsmath, amssymb, amsthm, mathrsfs}
\usepackage{bbm}
\usepackage{verbatim}
\usepackage{enumitem}
\usepackage[all]{xy}
\usepackage[pdftex]{color}

\numberwithin{equation}{section}
\newtheorem{thm}{Theorem}[section]

\newtheorem{lm}[thm]{Lemma}
\newtheorem{cor}[thm]{Corollary}
\theoremstyle{definition}
\newtheorem{df}[thm]{Definition}
\newtheorem{remark}[thm]{Remark}
\newtheorem{hyp}[thm]{Hypothesis}

\newtheorem{notn}[thm]{Notation}

\def\G{\mathrm{G}}
\def\Bset{\mathfrak S}
\def\Kset{\mathfrak S_K}

\def\tlG{\tilde\GF}
\def\hatG{\hat\GF}
\def\H{\mathrm{H}}
\def\GF{G}
\def\HF{H}
\def\MF{M}

\def\TF{T}

\def\gF{\mathfrak g}
\def\hF{\mathfrak h}
\def\mF{\mathfrak m}
\def\tF{\mathfrak t}

\DeclareMathOperator\Gal{Gal}

\DeclareMathOperator\Z{Z}

\DeclareMathOperator{\Hom}{Hom}
\DeclareMathOperator{\val}{val}

\def\M{\mathrm{M}}

\def\T{\mathrm{T}}

\DeclareMathOperator*{\Ad}{Ad}
\def\QQ{\mathbb{Q}}
\def\CC{\mathbb{C}}

\def\RR{\mathbb{R}}

\def\mcB{\mathcal B}

\def\mcO{\mathcal O}
\def\mcU{\mathcal U}

\def\mcZ{\mathcal Z}

\def\mcL{\mathcal L}

\def\mcE{\mathcal E}
\def\Kclass{\overline{\mathfrak s}}
\def\Bclass{\overline{\mathcal S}}
\def\oneK{\mathfrak s}
\def\oneB{\mathcal S}

\def\mO{\mathfrak O}

\def\charfn{\mathbbm{1}}
\DeclareMathOperator*{\Lie}{Lie}

\makeatletter
\newcommand\set{\@ifnextchar*\@shortset\@tallset}
\def\@shortset*#1#2{\{#1 \mathrel| #2\}}
\newcommand\@tallset[2]{
\left\{
\left.
	#1\vphantom{#2\bigl(\bigr)}\,
\right|
	\,#2
\right\}
}
\makeatother

\definecolor{red1}{rgb}{0.9, 0, 0}
\definecolor{red2}{rgb}{0.7, 0.3, 0.3}
\definecolor{blue1}{rgb}{0, 0, 1}
\definecolor{blue2}{rgb}{0.3, 0.3, 0.9}
\definecolor{green1}{rgb}{0, 0.2, 0}
\definecolor{green2}{rgb}{0, 0.4, 0}

\DeclareMathOperator\Ind{Ind}
\renewcommand\inf{\operatorname{inf}}
\renewcommand\P{\mathrm P}
\newcommand\res{\mathclose\rvert}
\DeclareMathOperator\vol{vol}
\newcommand\ceq{\mathrel{:=}}

\begin{document}

\title{The Bernstein projector determined by a weak associate class
of good cosets}

\author{Yeansu Kim}
\author{Loren Spice}
\author{Sandeep Varma}

\thanks{L.S.\ is supported in part by Simons grant 636151.}

\maketitle

\begin{abstract}
Let $\G$ be a reductive group over a $p$-adic field $F$ of characteristic
zero, with $p \gg 0$. In \cite{Kim04}, J.-L.~Kim studied an equivalence
relation called weak associativity
on the set of unrefined minimal $K$-types for $\G$ in the sense
of A.~Moy and G.~Prasad.
Following \cite{Kim04},
we attach to the set \(\Kclass\) of good \(K\)-types in a
weak associate class of positive-depth unrefined minimal
$K$-types a $\GF$-invariant open
and closed subset $\gF_{\Kclass}$ of the Lie algebra $\gF$ of
$\GF = \G(F)$, and a subset $\tlG_{\Kclass}$
of the admissible dual \(\tlG\) of \(\GF\)
consisting of those
representations containing an unrefined minimal $K$-type
that belongs to $\Kclass$.
Then \(\tlG_{\Kclass}\) is the union of finitely many
Bernstein components for $\G$, so that we can consider the
Bernstein projector $E_{\Kclass}$ that it determines.
We show that $E_{\Kclass}$ vanishes outside the Moy--Prasad $\GF$-domain
$\GF_r \subset \GF$, and reformulate a result of Kim
as saying that the restriction of $E_{\Kclass}$ to $\GF_r$, pushed
forward via the logarithm to the Moy--Prasad $\GF$-domain
$\gF_r \subset \gF$, agrees on $\gF_r$ with the inverse
Fourier transform of the characteristic function of $\gF_{\Kclass}$.
This is a variant of one of the descriptions given by
R.~Bezrukavnikov, D.~Kazhdan and Y.~Varshavsky
in \cite{BKV16} for the depth-$r$
Bernstein projector.
\end{abstract}

\section{Introduction}

Let $\G$ be a connected reductive group over a $p$-adic field $F$
of characteristic zero,
and put \(\GF = \G(F)\).
We use similar notation for other groups, writing, for example,
\(\HF = \H(F)\) (once we define \(\H\)).
Let $\tlG$ (respectively, $\hatG$) denote the set of isomorphism classes
of irreducible admissible (respectively, irreducible tempered) representations of $\GF$.
Throughout, we will use Fraktur letters to denote
the rational points of Lie algebras --- $\gF = \Lie(\G)(F)$, etc.
The terms undefined in the introduction are either standard
and can be found in the references we cite, or
are reviewed in Section \ref{sec: notn preliminaries}.

Recall that the set of $\GF$-conjugacy classes of
pairs $(\M, \sigma)$ consisting of a Levi subgroup $\M$ of
$\G$ and a (not necessarily unitary)
irreducible supercuspidal representation $\sigma$ of $\MF$, is (the
set of $\CC$-points of) a complex variety $\Omega(\G)$, called the
Bernstein variety of $\G$. The Zariski-connected components
of $\Omega(\G)$ are also its Hausdorff-connected components,
and by a Bernstein component of $\G$, we will refer
to the preimage of such a component in $\tlG$
under the infinitesimal character map
$\inf \colon \tlG \to \Omega(\G)$ that
sends any $\pi \in \tlG$ to the well defined
$\GF$-conjugacy class of its
supercuspidal support. We refer to \cite[Section 1.3]{BKV16}
for a review of general facts concerning
the Bernstein center $\mcZ(\G)$ of $\G$;
in Remark \ref{rmk: mcZ(G) review} we will further briefly review
a few that are particularly important for us.

Recall from \cite[Section 3.4]{MP96} the
notion of an (positive-depth) unrefined minimal $K$-type for $\G$.
If \(r > 0\) and \(x \in \mcB(\G)\), and
if \(\chi\) is a character of \(\GF_{x, r}\) that is trivial on \(\GF_{x, r+}\),
then combining Pontrjagin duality with
a Moy--Prasad isomorphism
\(\GF_{x, r}/\GF_{x, r+} \cong \gF_{x, r}/\gF_{x, r+}\)
considered in \cite[Sections 3.7 and 3.8]{MP94}
lets us attach to $\chi$
a certain subset
$Y + \gF^*_{x, (-r)+} \subset \gF^*$,
where $\gF^*$ is the $F$-vector
space dual to $\gF$.
Following \cite[Definition 2.1.2]{Kim04}, we will call
this subset $Y + \gF^*_{x, (-r)+}$ the dual
blob of the pair $\oneK = (\GF_{x, r}, \chi)$;
and we call \(\oneK\) an unrefined minimal \(K\)-type if its dual blob contains no nilpotent elements.
Further, an (positive-depth) unrefined
minimal $K$-type of the form $(\GF_{y, s}, \xi)$
is called associate to \(\oneK\)
if the dual blobs of $(\GF_{x, r}, \chi)$ and $(\GF_{y, s}, \xi)$
have \(\GF\)-conjugates that intersect non-trivially
\cite[Section 3.4]{MP96},
in which case automatically \(r = s\)
\cite[Corollary 3.2.2]{AD02}.
Following \cite[Definition 2.2.1]{Kim04}, let `weak associativity'
be the transitive closure of the relation of associativity
on the set of positive-depth, unrefined minimal $K$-types for $\G$.

We warn the reader that, in
the body of this paper, we work with a
notion of weak associativity
that (slightly and, under our hypotheses, harmlessly) differs from this one.
Namely, we give a definition that turns out to be the restriction to good unrefined minimal \(K\)-types, in the sense of \cite[Definition 3.2.1]{Kim04}, of the definition used in the introduction.
See Subsubsection \ref{subsubsec: minimal K type weak associativity}.
Accordingly, we write \(\Kclass\) for the set of good unrefined minimal \(K\)-types that are weakly associate to \(\oneK\),
so that the body of the paper will call \(\Kclass\) a weak associate class.

Write $\tlG_\oneK$ for the subset of $\tlG$
formed of (the isomorphism classes of) all the irreducible admissible
representations of $\GF$ containing a (not necessarily good) weak associate
of $\oneK$,
and \(\tlG_{\Kclass}\) for the analogous subset of representations containing an element of \(\Kclass\).
Remark \ref{rmk: GS GmcS} uses \cite[Theorem 4.5.1]{Kim04} to show that these sets are equal.
Lemma \ref{lm: types are types} shows that
\(\tlG_\oneK = \tlG_{\Kclass}\) is a union of Bernstein components of $\G$,
so one can consider the element
$E_\oneK = E_{\Kclass} \in \mcZ(\G)$ that acts as the identity on
elements of $\tlG_{\Kclass}$, and as zero on the elements
of $\tlG \setminus \tlG_{\Kclass}$.

In \cite{BKV16}, for a nonnegative real number $s$,
R.~Bezrukavnikov, D.~Kazhdan and
Y.~Varshavsky study the depth-$s$ Bernstein projector
for $\G$.  This is the element $E_s \in \mcZ(\G)$
that acts as the identity on every representation
$\pi \in \tlG$ of depth at most $s$, and as zero
on every representation $\pi \in \tlG$ of
depth greater than $s$. They
give two descriptions for $E_s$. The
first is an Euler--Poincar\'e formula for $E_s$
\cite[Theorem 1.6]{BKV16}.
The second involves
the Lie algebra $\gF$: they show
\cite[Corollary 1.22]{BKV16}
that
$E_s$ is supported on the Moy--Prasad $\GF$-domain $\GF_{s+}$, and that,
if there is an `$s$-logarithm' map $\mcL \colon \GF_{s+} \to \gF_{s+}$,
then the push-forward $\mcL_!(E_{s+})$ of $E_{s+}$,
which is a distribution on $\gF$ supported in $\gF_{s+}$, equals the
(suitably normalized) inverse Fourier transform of the
characteristic function $\charfn_{\gF^*_{-s}}$ of the
Moy--Prasad $\GF$-domain $\gF^*_{-s} \subset \gF^*$.

Thus, it is natural to ask if one could refine the above description
of $E_s$ into a description of $E_{\Kclass}$.
We do not give an Euler--Poincar\'e description of \(E_{\Kclass}\).
However,
under Hypotheses \ref{hyp: Kim Spi}, \ref{hyp: Spi 1}, and \ref{hyp: Spi 2},
we do describe $E_{\Kclass}$
in terms of an inverse Fourier transform of a characteristic function
$\charfn_{\gF_{\Kclass}}$, where $\gF_{\Kclass} \subset \gF$ is a subset
defined by Ju-Lee Kim; see Subsubsection \ref{sssec: weak associativity remarks}
for this notation.
(There appears to be a difference here, since \cite{BKV16} deals with a subset of \(\gF^*\) but we deal with a subset of \(\gF\); but our Hypothesis \ref{hyp: Kim Spi} affords an identification of \(\gF\) with \(\gF^*\), so really this is just a cosmetic difference.)

In Subsection \ref{subsec: cosets dual blobs}, we impose Hypothesis \ref{hyp: Kim Spi}\ref{hyp: Kim 1.3.6}, which, among other things, guarantees that, for every \(s > 0\), the logarithm map is a \(\GF\)-equivariant, \(p\)-adic analytic isomorphism \(\log \colon \GF_s \to \gF_s\).
Thus, we can speak
of the push-forward $\log_!(E_{\Kclass}\res_{\GF_r})$ of $E_{\Kclass}$ along
$\log$ as an invariant distribution on the $\GF$-domain
$\gF_r \subset \gF$.

Our main result is a description of the projectors $E_{\Kclass}$.
\begin{thm} \label{thm: main}
Let $\Kclass$ be the set of good unrefined minimal
$K$-types in a weak associate class of unrefined
minimal $K$-types of positive depth $r > 0$,
and let $\Gamma \in - \gF_{\Kclass}$ be a good semisimple element.
Assume Hypothesis \ref{hyp: Kim Spi} for the positive
real number $r$, as well
as Hypotheses \ref{hyp: Spi 1} and \ref{hyp: Spi 2} for
$\Kclass$ and $\Gamma \in - \gF_{\Kclass}$. Then:
\begin{enumerate}[label=(\roman*)]
\item\label{thm: E support} $E_{\Kclass}$ is supported on $\GF_r$.
\item\label{thm: log E} The transfer $\log_!(E_{\Kclass}\res_{\GF_r})$
of $E_{\Kclass} \res_{\GF_r}$ to $\gF_r$
coincides with the restriction to $\gF_r$ of the
inverse Fourier transform
of $\charfn_{\gF_{\Kclass}}$, where the Fourier transform is normalized
as in Subsection \ref{subsec: Fourier transform definition}.
\end{enumerate}
\end{thm}

It turns out that the strategy for our proof of
Theorem \ref{thm: main}\ref{thm: E support} is a modification
of the strategy of the proof of \cite[Theorem 3.3.1]{Kim04}:
while that proof studies
certain invariant distributions using the character expansion
of \cite{KM03}, which is valid near the identity element,
our proof of Theorem \ref{thm: main}\ref{thm: E support} uses
\cite[Theorem 4.4.11]{Spi18}, which
gives a character expansion in the spirit of \cite{KM03}
about semisimple elements of $\GF$ other than the identity.

As we were finishing up our work, we noticed the recent
preprint \cite{MS20} by A.~Moy and G.~Savin,
which gives an Euler--Poincar\'e formula for
a Bernstein projector that is very closely related
to (possibly the same as or a bit finer than)
the one that we are considering. Their work does not
require the strong assumptions on $p$ that we are
imposing. However, apart from the fact that the description
that we offer is different from theirs,
our work also introduces the idea of using character expansions
to study Bernstein projectors, making our proof
different from the one
in \cite{MS20}, shedding light on these projectors
from a different angle.

Our work also offers a glimpse of
a possible version of Theorem \ref{thm: main}
where good unrefined minimal $K$-types are replaced by
the `slightly refined' minimal $K$-types
$(K_{\Sigma}^+, \phi_{\Sigma})$ from
\cite[Section 4.2.2]{KM06}, that were studied by
by J.-L.~Kim and F.~Murnaghan in \cite{KM06}
and \cite{Kim07}.  In fact, using \cite{KM06} in
place of \cite{KM03} and \cite[Equation (0.2)]{Kim07}
in place of \cite[Theorem 3.3.1]{Kim04}
should allow one to generalize Theorem \ref{thm: main}\ref{thm: log E}
to the setting of the $(K_{\Sigma}^+, \phi_{\Sigma})$,
provided one also proves the analogue of
Lemma \ref{lm: types are types} in this `slightly refined' situation.
Generalizing Theorem \ref{thm: main}\ref{thm: E support}
to the $(K_{\Sigma}^+, \phi_{\Sigma})$ would take more work,
since the results from \cite{Spi18} that we use haven't yet been
adapted to deal with the $(K_{\Sigma}^+, \phi_{\Sigma})$.
Such a generalization would be a nice way to think
of J.-L.~Kim's approach in \cite{Kim07} to the exhaustiveness
of J.-K.~Yu's construction of supercuspidal representations.
We hope to return to this question
in future work.

\section{Some notation and preliminaries} \label{sec: notn preliminaries}
We fix an algebraic closure $\bar F$ of $F$, and let
$\val \colon \bar F \to \QQ \cup \{\infty\}$ be
the usual extension to $\bar F$ of the normalized
discrete valuation on $F$. Throughout the rest of this paper,
Hypothesis \ref{hyp: Kim Spi}, which includes
\cite[Hypothesis 1.3.6]{Kim04}, will be in force.

\subsection{Bruhat--Tits buildings and Moy--Prasad filtrations}
\label{subsec: mcB embeddings}

Let $\mcB(\G)$ denote the enlarged Bruhat--Tits building of $\G$.
For $x \in \mcB(\G)$ and $r \geq 0$ (respectively, $r \in \mathbb{R}$),
we have the Moy--Prasad subgroups $\GF_{x, r}, \GF_{x, r+} \subset \GF$
(respectively, the Moy--Prasad lattices $\gF_{x, r}, \gF_{x, r+} \subset \gF$). For $r \geq 0$ (respectively, $r \in \mathbb{R}$),
we have Moy--Prasad $\GF$-domains $\GF_r, \GF_{r+} \subset \GF$
(respectively, $\gF_r, \gF_{r+} \subset \gF$) defined by:
\begin{equation*} \label{eqn: Moy-Prasad domains}
\GF_{r} \ceq \bigcup_{x \in \mcB(\G)} \GF_{x, r},\quad
\GF_{r+} \ceq \bigcup_{x \in \mcB(\G)} \GF_{x, r+},\quad
\gF_{r} \ceq \bigcup_{x \in \mcB(\G)} \gF_{x, r},\quad
\gF_{r+} \ceq \bigcup_{x \in \mcB(\G)} \gF_{x, r+}.
\end{equation*}
Obvious modifications of this notation will apply to other groups:
thus, if $\H$ is a reductive group, we will talk
of $\HF_r$, $\hF_{r+}$, etc.

Let \(\G'\) be a reductive subgroup of \(\G\) containing a maximal torus \(\T\) of \(\G\) that splits over a tame extension \(L\) of \(F\)
(which every torus in \(\G\) will, once we impose Hypothesis \ref{hyp: Kim Spi}\ref{hyp: Kim 1.3.6}).
Then \cite[Proposition 4.6]{AS08} constructs a family
of embeddings $\mcB(\G') \hookrightarrow \mcB(\G)$
of the Bruhat--Tits building of $\G'$ into
that of $\G$, with the property that for every such
embedding $\iota$, every $x \in \mcB(\G')$ and every $r > 0$,
we have
$\GF'_{x, r} = \GF' \cap \GF_{\iota(x), r}$.
For this subsection alone, we informally refer
to the elements of such a family of embeddings
$\mcB(\G') \hookrightarrow \mcB(\G)$ as
`canonical'.
In fact, the canonical embedding constructed by \cite[Proposition 4.6]{AS08}
is part of a larger collection
of analogous embeddings over discretely valued, separable extensions of \(F\), but these extra data will not concern us,
except that we use below the embedding \(\mcB(\G', L) \to \mcB(\G, L)\) corresponding to the splitting field \(L\) of \(\T\)
to show that our notion of canonical embedding agrees with that of \cite[Section 1.3]{Kim04}.
These embeddings are compatible in the sense that the precomposition of a canonical embedding \(\mcB(\G') \hookrightarrow \mcB(\G)\) with conjugation by an element \(g \in \GF\) is a canonical embedding \(\mcB(g^{-1}\G' g) \hookrightarrow \mcB(\G)\); and, if \(\H'\) is a reductive subgroup of \(\G'\) containing a maximal torus of $\G'$, hence also a reductive subgroup of \(\G\) containing a maximal torus of $\G$, then the composition of any of the canonical embeddings \(\mcB(\H') \hookrightarrow \mcB(\G')\) and any of the canonical embeddings \(\mcB(\G') \hookrightarrow \mcB(\G)\) is one of the canonical embeddings \(\mcB(\H') \hookrightarrow \mcB(\G)\).

Any two canonical embeddings differ
in a well understood way \cite[Remark 4.7]{AS08}.
In particular, they all have the same image, so we may, and do,
consider \(\mcB(\G')\) as a subset of \(\mcB(\G)\).

When $\G'$ is a twisted Levi subgroup of $\G$, these embeddings are also discussed by \cite[Section 1.3]{Kim04}, which references \cite[Remark 2.11]{Yu01}.  Since the explicit description of the embeddings in \cite[Remark 2.11]{Yu01} and \cite[Proposition 4.6]{AS08} work by tame descent from the splitting field of \(\T\), it suffices to work with \(\mcB(\G', L) \hookrightarrow \mcB(\G, L)\), and so assume that \(\T\) is actually a \emph{split} maximal torus.  Then the explicit apartment-by-apartment descriptions in the two references show that the notions of canonical embedding agree.

\subsection{An additive character
and the Fourier transform}
\label{subsec: Fourier transform definition}

For the remainder of the paper,
we fix a continuous character $\Lambda \colon F \to \CC^{\times}$
that is nontrivial on the ring $\mO$ of integers of $F$,
but trivial on the maximal ideal of \(\mO\).
Let \(V\) be any \(F\)-vector space.
Then our choice of \(\Lambda\) gives an isomorphism between the Pontrjagin dual \(\Hom_{\text{cts}}(V, \CC^\times)\) and the linear dual \(V^* = \Hom_F(V, F)\),
so further choosing a Haar measure \(dY\) on \(V\) defines
a Fourier transform
$C_c^{\infty}(V) \to C_c^{\infty}(V^*)$ by
$f \mapsto \hat f$, where
\[ \hat f(Y) = \int_V f(X) \Lambda(\langle Y, X\rangle) \, dX. \]
If $T$ is a distribution on \(V^*\)
(i.e., a linear map \(C_c^\infty(V^*) \to \CC\)),
then we define
its Fourier transform $\hat T \colon C_c^{\infty}(V) \to \CC$
to be the distribution given by $\hat T(f) = T(\hat f)$.

When we apply this definition, \(V = \hF\) will be the space of \(F\)-rational points of the Lie algebra of an algebraic \(F\)-group \(\H\) for which we can use Hypothesis \ref{hyp: Kim Spi}\ref{hyp: Kim 1.3.6} to identify \(\hF \cong \hF^*\), and so to regard the Fourier transform \(\hat f\) of a function \(f\) on \(\hF\) as again a function on \(\hF\), and hence to define the Fourier transform of a distribution on \(\hF\) as again a distribution on \(\hF\).

\subsection{Good cosets, weak associativity and a partition of $\gF$}
\label{subsec: good cosets etc.}

\subsubsection{Good elements}
Recall from \cite[Definition 1.2.1]{Kim04} that for
$s \in \RR$, a good
element of depth $s$ is a tame element $\Gamma \in \gF$
(i.e., an element that belongs to the Lie algebra of some tamely ramified maximal torus
--- which, once we impose Hypothesis \ref{hyp: Kim Spi}\ref{hyp: Kim 1.3.6},
will be any semisimple element)
such that, for some (equivalently, every)
maximal torus
$\T$ with $\Gamma \in \tF$, we have that $\Gamma \in \tF_{s} \setminus \tF_{s+}$,
and such that for every absolute root $\alpha$ of \(\T\) in \(\G\),
$d\alpha(\Gamma)$ is either zero or has valuation $s$.
(The latter condition \emph{almost} implies the former, in the sense that, if it is satisfied, then \(\Gamma\) does not belong to \(\tF_{s+}\) unless \(\Gamma\) is centralized by \(\G\), and \(\Gamma\) belongs to \(\tF_s\) if \(\G\) is adjoint.)

\subsubsection{Good cosets and their weak associativity}
\label{subsubsec: good cosets weak associativity}
Let $r > 0$. By a good coset of depth $-r$,
we mean a set of the form $\Gamma + \gF_{x, (-r)+}$,
where $\Gamma \in \gF$ is a good element of depth $-r$, and
$x \in \mcB(\G') \subset \mcB(\G)$, where $\G'$ is the
centralizer of $\Gamma$ in \(\G\). By \cite[Theorem 2.3.1]{KM03},
this agrees with the marginally
different formulation of the same notion in \cite[Definition 1.2.2]{Kim04}.

Two good cosets are said to be associate
to each other if they have \(\GF\)-conjugates that intersect non-trivially,
in which case their depths are the same \cite[Remark 2.4.2(1)]{KM03}.
The transitive closure of the relation of associativity
is an equivalence relation, called weak associativity.
(The definition in \cite[Definition 1.4.1]{Kim04}
is a priori coarser,
in the sense that it might call more good cosets weakly associate,
because the chain of associate non-degenerate cosets linking them is not required to consist only of good cosets.
Actually, this does not happen, thanks to
\cite[Lemma 1.4.2]{Kim04}, which implies that two
good cosets that are weakly associate in the more general sense
of \cite[Definition 1.4.1]{Kim04}
are associate to a common good coset, hence weakly associate in our sense.)

The notions of good coset,
and of weak associativity for good cosets,
are also defined in
\cite{Kim04} for $r = 0$, but we will avoid
referring to these notions.

\subsection{Unrefined minimal cosets and dual blobs}
\label{subsec: cosets dual blobs}

We now need to impose Hypothesis \ref{hyp: Kim Spi}, in order to speak of dual blobs.
This hypothesis depends on a positive real number \(r\), which we will take to be the depth of a weak associate class \(\Kclass\) of good unrefined minimal \(K\)-types in Subsection \ref{subsec: projector}.
We will later add Hypotheses \ref{hyp: Spi 1} and \ref{hyp: Spi 2}.

\begin{hyp} \label{hyp: Kim Spi}\hfill
\begin{enumerate}[label=(\roman*)]
\item\label{hyp: Kim 1.3.6}
The group \(\G\) satisfies \cite[Hypothesis 1.3.6]{Kim04}
(consisting of hypotheses labeled (HB), (HGT) and (H\(k\)) there),
where the reference to the adjoint representation in (H\(k\)) is replaced by a faithful representation of \(\G\), as in \cite[Section 3.1.0]{Kim99} and \cite[Appendix B]{DR09}.  This (together with the fact that there are finitely many rational conjugacy classes of maximal tori in \(\G\)) guarantees the existence of a finite, tame, Galois extension \(L_\G\) of \(F\) that splits every maximal torus in \(\G\),
and we require also that the analogue of \cite[(H\(k\))]{Kim04} is satisfied for the base change of \(\G\) to \(L_\G\), with respect to the same faithful representation defined over $F$.
\item\label{hyp: log}
For all $x$ belonging to the
Bruhat--Tits building $\mcB(\G)$ of $\G$,
the logarithm map with respect to our chosen faithful representation of \(\G\)
(which, from \ref{hyp: Kim 1.3.6},
restricts to an analytic (but not group-theoretic) isomorphism
$\G(L_\G)_{x, r} \to \Lie(\G)(L_\G)_{x, r}$)
carries cosets of \(\G(L_\G)_{x, r+}\) to cosets of \(\Lie(\G)(L_\G)_{x, r+}\), and the induced map
\[
\overline\log \colon \G(L_\G)_{x, r}/\G(L_\G)_{x, r+} \to
\Lie(\G)(L_\G)_{x, r}/\Lie(\G)(L_\G)_{x, r+}
\]
is an isomorphism of abelian groups.
\item\label{hyp: p big}
$p > e(L_{\G}/\QQ_p)/(p - 1)$, where $L_{\G}$ is as in
\ref{hyp: Kim 1.3.6} and $e(L_{\G}/\QQ_p)$ denotes its ramification degree
over $\QQ_p$.
\end{enumerate}
\end{hyp}

\begin{remark}\hfill
\label{rmk: hyp Kim Spi consequences}
\begin{itemize}
\item Hypothesis \ref{hyp: Kim Spi}
is satisfied whenever the residue characteristic $p$ is greater
than a constant that depends only on the absolute root
datum of $\G$ and the degree of ramification of $F$ over
$\QQ_p$.
To avoid a lengthy digression here, we do not go into details; but, for some conditions that imply parts of Hypothesis \ref{hyp: Kim Spi}, see \cite[Proposition 4.1]{AR00}, \cite[Corollary 2.6 and Theorem 3.3]{Fin19}, \cite[Proposition 3.1.1]{Kim99}, and \cite[Lemma B.5.4 and Lemma B.6.12]{DR09}
(noting that the proof of \cite[Lemma B.7.2]{DR09} also shows that, in the notation of that result, \(Z_r \times G'_r \to G_r\) is a bijection for all \(r > 0\)).
\item One way that we use Hypothesis \ref{hyp: Kim Spi}\ref{hyp: Kim 1.3.6} is to ensure the existence of Moy--Prasad isomorphisms, which do not exist in general.  See, for example, \cite[remark following Corollary 5.6]{Yu15}.
\item Hypothesis \ref{hyp: Kim Spi}\ref{hyp: Kim 1.3.6}
says that there exists a symmetric
nondegenerate $\Ad \GF$-invariant bilinear form $B$ on
$\gF$ such that the resulting isomorphism from $\gF$ to
its dual vector space $\gF^*$ identifies, for each
$x \in \mcB(\G)$ and each $r \in \RR$, the Moy--Prasad lattice
$\gF_{x, r}$ with
\[
\gF^*_{x, r} \ceq \set
	{Y \in \gF^*}
	{\text{\(\langle Y, \gF_{x, (-r)+} \rangle\)
		is contained in the maximal ideal of \(\mO\)}}.
\]
If \(\H\) is a reductive subgroup of \(\G\) containing a maximal torus in \(\G\)
(for example, \(\H = \G\)), then the restriction of \(B\) to \(\hF\) is non-degenerate, hence also furnishes an isomorphism \(\hF \cong \hF^*\).  This allows us to view the Fourier transform as a map \(C_c^\infty(\hF) \to C_c^\infty(\hF)\), and hence, in particular, to iterate it.  We will always equip \(\hF\) with the unique Haar measure, called self dual, for which \(\Hat{\Hat f}(Y) = f(-Y)\) for all \(f \in C_c^\infty(\hF)\) and all \(Y \in \hF\).
\item By \cite[Corollary 2.3]{Yu01} and the fact that \(\log\) is \(\Gal(L_\G/F)\)-equivariant, Hypothesis \ref{hyp: Kim Spi}\ref{hyp: log}, which is a statement about \(L_\G\)-rational points, implies the analogous statement about \(F\)-rational points.
\end{itemize}
\end{remark}

Compare Lemma \ref{lm: H domain} and Corollary \ref{cor: H domain} to \cite[Hypothesis 8.3]{AK07}.
Corollary \ref{cor: H domain} is a strictly stronger result, but, since its proof, which uses the exponential map, is different in spirit from that of Lemma \ref{lm: H domain}, we have separated them.

\begin{lm}
\label{lm: H domain}
If \(\H\) is a connected, reductive subgroup of \(\G\), then
\(\hF \cap \gF_s\) equals \(\hF_s\) for all real numbers \(s\),
and
\(\HF \cap \GF_s\) equals \(\HF_s\) for all \(s > 0\).
\end{lm}

\begin{proof}
By \cite[Lemma 3.3.8 and Lemma 3.7.18]{AD02}, it suffices to show that the sets of semisimple elements in \(\HF \cap \GF_s\) and \(\HF_s\),
and in \(\hF \cap \gF_s\) and \(\hF_s\), are the same.
Let us first prove the `group case', i.e.,
fixing a semisimple element \(h \in \HF\), let us
see that $h \in \GF_s$ if and only if $h \in \HF_s$.
Since $\H$ is connected, we can
choose a maximal torus $\T_{\H}$ in \(\H\) containing \(h\).
Let \(\T\) be a maximal torus in \(\G\) containing \(\T_\H\).
Note that \(\T \cap \H\) is contained in, hence equals, \(\mathrm C_\H(\T_\H) = \T_\H\).
By Hypothesis \ref{hyp: Kim Spi}\ref{hyp: Kim 1.3.6}, we have that \(\T\) splits over a tame extension of \(F\).
By \cite[Lemma 2.2.3]{AD04}, we may, and do, replace \(F\) by \(L_\G\), so that \(\T\), hence also \(\T_\H\), is split.
Now \cite[Lemma 2.2.9]{AD04} gives that \(\TF_\H \cap \HF_s\) equals \((\TF_\H)_s\) and \(\TF \cap \GF_s\) equals \(\TF_s\); and, since \(\T_\H\) and \(\T\) are split and every (algebraic) character of \(\T_\H\) extends to one of \(\T\), we see that
\[
(\TF_\H)_s = \set{t \in \TF_\H}
	{\text{\(\chi_\H(t) - 1\) has valuation at least \(s\) for all characters \(\chi_\H\) of \(\T_\H\)}}
\]
is the intersection with \(\HF\) of
\[
\TF_s = \set{t \in \TF}
	{\text{\(\chi(t) - 1\) has valuation at least \(s\) for all characters \(\chi\) of \(\T\)}}.
\]
The result for the group follows.
The proof in the Lie-algebra case is similar, replacing \cite[Lemma 2.2.3 and Lemma 2.2.9]{AD04} by \cite[Lemma 2.2.5 and Corollary 2.2.7]{AD04}.
\end{proof}

\begin{cor} \label{cor: H domain}
If \(\H\) is a possibly disconnected reductive subgroup of \(\G\),
then we have for all real numbers \(s\) that
\(\hF \cap \gF_s = \hF_s\);
and for all \(s > 0\) that
\(\HF \cap \GF_s = \HF_s\)
and $\log \colon \GF_s \to \gF_s$
and \(\exp \colon \gF_s \to \GF_s\)
restrict to mutually inverse \(p\)-adic analytic isomorphisms
$\HF_s \leftrightarrow \hF_s$
\end{cor}

\begin{proof}
By Lemma \ref{lm: H domain} applied with the
identity component $\H^{\circ}$ of $\H$ in place of $\H$,
we have that
$\hF \cap \gF_s
= \Lie(\H)(F) \cap \gF_s
= \Lie(\H^\circ)(F) \cap \gF_s
= \Lie(\H^\circ)(F)_s
= \hF_s$ for all $s \in \RR$ and
\(\H^\circ(F) \cap \GF_s
= (\H^\circ)(F)_s
= \HF_s\) for all \(s > 0\).
Thus, it remains to show that the inclusion
$\H^{\circ}(F) \cap \GF_s \subset \HF \cap \GF_s$ is an equality
for all $s > 0$.
We slightly abuse notation by writing \(\HF^\circ\) for \(\H^\circ(F)\).

We recall that Hypothesis \ref{hyp: Kim Spi}\ref{hyp: Kim 1.3.6} does
not depend on $r$, and gives us mutually inverse,
\(\GF\)-equivariant,
$p$-adic analytic maps
\(\exp \colon \gF_s \to \GF_s\)
and
$\log \colon \GF_s \to \gF_s$.
By \cite[Corollary 3.8]{Bor98},
the subset
$\log(\HF \cap \GF_s) \subset \gF_s$
lies in \(\hF\), hence in
$\hF \cap \gF_s = \hF_s$.
Thus, \(\log^{-1}(\hF_s)\) contains \(\HF \cap \GF_s\).
It now suffices to show that
$\exp(\hF_s)$ is contained in $\HF^{\circ}$,
so that \(\log^{-1}(\hF_s)\) is contained in \(\HF^\circ\),
and hence \(\HF \cap \GF_s \subset \log^{-1}(\hF_s)\)
is contained in \(\HF^\circ \cap \GF_s = \HF_s\).
\newcommand\semi{_{\text{ss}}}%
\newcommand\nil{_{\text{nil}}}%
Fix an element \(X\) of \(\hF_s\).
By \cite[Theorem 3.1.2 and Lemma 3.3.8]{AD02},
the semisimple parts \(X\semi\) and \(X\nil\) 
of \(X\) also belong to \(\hF_s \subset \gF_s\).
Then $\exp(X) = \exp(X\semi)\exp(X\nil)$.
Because $\exp(X\nil) \in \HF$ is unipotent,
it is contained in $\HF^{\circ}$ (because \(F\) has characteristic \(0\)).
We thus may, and do, replace \(X\) by its semisimple part.
Then there is some maximal torus \(\T_\H\) in \(\H^\circ\)
such that \(X\) lies in \(\tF_\H\).
By \cite[Corollary 3.8]{Bor98} again, we have that
\(\exp(X)\) lies in \(\TF_\H\),
which is contained in \(\HF^\circ\).
\end{proof}

\begin{remark}
Let \(\T\) be a maximal torus in \(\G\) such that \(x\) lies in \(\mcB(\T)\).  We claim that the map \(\overline\exp = \smash{\overline\log}^{-1}\) agrees with the map \(\varphi_{T, x; r, r+}\) of \cite[Section 1.5, pages 12--13]{Adl98}.
(Strictly speaking, for consistency with \cite{Adl98}, we should instead write \(\varphi_{T, \{x\}; r, r + \varepsilon}\) for \(\varepsilon > 0\) sufficiently small, but we will not do do so.)
In fact, this property is part of why those maps are called mock exponential.

Since \(\varphi\) is defined by (tame) descent, it suffices by \cite[Corollary 2.3]{Yu01} to replace \(F\) by \(L_\G\) and \(\G\) by its base change, and so assume that \(\T\) is split.
For each absolute root \(b\) of \(\T\) in \(\Lie(\G)\), we have that \(\exp\) is actually a homomorphism of algebraic groups \(\Lie(\G)_b \to \G\).  It is \(\T\)-equivariant, and hence has image in the corresponding root subgroup.  Since its derivative is the identity, it agrees with the maps \(\exp_b\) of \cite[Section 1.5, page 11]{Adl98}.
Further, Hypothesis \ref{hyp: Kim Spi}\ref{hyp: p big} implies that, for all $t$ of positive valuation in $F$,
the series \(\tau = \sum \frac1{i!}t^i\) converges, and the valuation of
$\tau - 1 - t$ is greater than that of \(t\).
Thus, the restriction of \(\overline\exp\) to \(\tF_r/\tF_{r+}\)
agrees with the map \(\varphi_{T; r, r+} \colon \tF_r/\tF_{r+} \to \TF_r/\TF_{r+}\) of \cite[Section 1.5, page 11]{Adl98}.
Since \(\overline\exp\) and \(\varphi_{T, x; r, r+}\) are both homomorphisms, and since they agree on a set of generators for their domain, it follows that they agree everywhere.
\end{remark}

\begin{notn} \label{notn: LambdaX}
Fix \(X \in \gF\).
\begin{enumerate}[label=(\roman*)]
\item Let $\Lambda_X$ be the character
$\gF \to \CC^{\times}$ given by $Z \mapsto \Lambda(B(X, Z))$.
\item\label{notn: LambdaX group} Given $x \in \mcB(\G)$ and $r > 0$,
let $\Lambda_{X, x, r}$ be
the function $\GF_{x, r} \to \CC^{\times}$
given by $g \mapsto \Lambda(B(X, \log g))$.
If $X \in \gF_{x, -r}$, then Hypothesis \ref{hyp: Kim Spi}\ref{hyp: log}
shows that \(\Lambda_{X, x, r}\) is a character, and
allows us to think of $\Lambda_{X, x, r}$ as
the map $g \mapsto \Lambda(B(X, \overline{\log}(g)))$ on
$\GF_{x, r}/\GF_{x, r+}$.
\end{enumerate}
\end{notn}

\begin{notn}
\label{notn: dual blob}
Suppose $K \subset \GF$ is a compact open subgroup
of the form $\exp(L)$, for a lattice $L \subset \gF$.
Let $\chi \colon K \to \CC^{\times}$ be a continuous
character.
Then the dual blob of $\chi$ or of $(K, \chi)$
is defined to be the set of all
$Y \in \gF$ such that $\chi$ equals $\Lambda_Y \circ \log\res_K$,
provided this set is nonempty.
\end{notn}

Note that if the dual blob of $\chi$ exists, then it is a single coset
in $\gF$ for the kernel $L^{\perp}$ of the homomorphism
$Y \mapsto \Lambda_Y\res_L \in \Hom(L, \CC^{\times})$,
so this notion is in agreement
with \cite[Definition 3.10]{Kim07}.

Recall that, if $r > 0$, then an unrefined
minimal $K$-type of depth $r$ for $\G$ in the
sense of \cite[Section 3.4]{MP96} is a pair
$(\GF_{x, r}, \chi)$, where \(\chi\) is a character of $\GF_{x, r}/\GF_{x, r+}$
(or, by inflation, of \(\GF_{x, r}\))
of the form $\Lambda_{Y, x, r}$ as in
Notation \ref{notn: LambdaX}\ref{notn: LambdaX group}, for some
$x \in \mcB(\G)$ and $Y \in \gF_{x, -r}$, with the property that the coset
$Y + \gF_{x, (-r)+}$ contains
no nilpotent element (i.e., is non-degenerate in the sense of \cite[Section 3.4]{MP96}).
In this case, the coset $Y + \gF_{x, (-r)+}$
is equal to the dual blob of \((\GF_{x, r}, \chi)\)
(and in particular depends on the triple $(\G, x, r)$
only through the subgroup $\GF_{x, r}$).

\subsubsection{Good unrefined minimal $K$-types and
their weak associativity} \label{subsubsec: minimal K type weak associativity}

We now recall notions for unrefined minimal \(K\)-types parallel to the notions for cosets introduced in Subsubsection \ref{subsubsec: good cosets weak associativity}.
Recall that there is a depth-reversing bijection between positive-depth unrefined minimal \(K\)-types and negative-depth, non-degenerate Moy--Prasad cosets that sends each type to its dual blob, i.e.,
\((\GF_{x, r}, \Lambda_{X, x, r})
\mapsto X + \gF_{x, {(-r)+}}\).
We say (following \cite[Definition 3.2.1]{Kim04})
that an unrefined minimal $K$-type
of positive depth is good if its dual blob is a good coset,
so that good (positive-depth) unrefined minimal \(K\)-types are in bijection with good (negative-depth) cosets, and then use the bijection to transport the notion of weak associativity from the latter to the former.
(In particular, since weakly associate good cosets have the same (negative) depth, weakly associate good \(K\)-types have the same (positive) depth.)
As in Subsubsection \ref{subsubsec: good cosets weak associativity},
although our definition might a priori distinguish good \(K\)-types that are equivalent according to \cite[Definition 2.2.1]{Kim04} (which allows a chain of associated \(K\)-types that includes non-good \(K\)-types), actually this does not happen.

The reason we find it convenient to
restrict to good unrefined minimal cosets, unlike
\cite{Kim04}, is that we are able to appeal to one of the culminating results of that paper,
\cite[Theorem 4.5.1]{Kim04} (or \cite[Theorem 2.4.10]{KM03}),
which guarantees the existence of a good unrefined minimal
$K$-type in every element of $\tlG$.

\subsubsection{Remarks on the sets of weak associate
classes of good unrefined minimal
$K$-types and good cosets}
\label{sssec: weak associativity remarks}

Let $\Kset$ be the disjoint union of a singleton set
$\{*\}$ and the set of weak associate
classes of good positive-depth unrefined
minimal $K$-types. Let $\Bset$ be the disjoint union
of a singleton set $\{*'\}$ and the set of all weak associate
classes of negative-depth good cosets.
One thinks of \(*\) as standing for the collection of all depth-\(0\) \(K\)-types, in the sense of \cite[Section 3.4]{MP96}, and of \(*'\) as standing for \(\gF_0\); this is in accord with the depth-\(0\) case of \cite[Definition 2.1.1]{Kim04}.
Passage to weak associate classes affords
a well defined bijection $\iota \colon \Kset \to \Bset$
such that $\iota(*) = *'$ and $\iota(\Kclass) = \Bclass$
whenever there are $\oneK \in \Kclass$ and \(\oneB \in \Bclass\)
such that the dual blob of \(\oneK\) is \(\oneB\)
(in which case the dual blob of every element of \(\Kclass\) belongs to \(\Bclass\)).
Recall that \(\iota\) is depth-reversing, in the sense that, if some \(\oneK\) in \(\Kclass\) has depth \(r\), then, for every coset \(\oneB\) in \(\iota(\Kclass)\), every element of \(\oneB\) has depth \(-r\).

For each weak associate class $\Bclass \in \Bset \setminus \{*'\}$
of negative-depth good cosets,
\cite[Definition 1.4.3]{Kim04}
defines the subset $\gF_{\Bclass} \subset \gF$ to be the union of
all $\Ad \GF(\oneB)$, as $\oneB$ ranges over $\Bclass$. If
$\Bclass = *' \in \Bset$, define $\gF_{\Bclass}$ to
be the Moy--Prasad $\GF$-domain $\gF_0 \subset \gF$.
Given $\Kclass \in \Kset$, we define $\gF_{\Kclass}$ to
be $\gF_{\Bclass}$, where $\Bclass = \iota(\Kclass)$.

Note that, given $\Kclass \in \Kset \setminus \{*\}$,
the set $\Kclass^{\vee}$
consisting of the contragredients of the members of
$\Kclass$ also belongs to $\Kset \setminus \{*\}$. Moreover
\[ \iota(\Kclass^{\vee}) = - \iota(\Kclass)
= \set{- \oneB}{\oneB \in \iota(\Kclass)}. \]

We conclude that $\gF_{\Kclass^{\vee}} = - \gF_{\Kclass}$ for
every weak associate class $\Kclass$ of good positive-depth unrefined
minimal $K$-types.

\subsubsection{Subsets of $\gF$, $\tlG$ and $\hatG$
determined by a good unrefined minimal $K$-type}
\label{sssec:subsets}

By \cite[Lemma 1.4.5]{Kim04}, we have that
\begin{equation*} \label{eqn: g(F) unrefined good minimal decomposition}
\gF = \bigsqcup_{\Bclass \in \Bset} \gF_{\Bclass},
\end{equation*}
and, since each $\gF_{\Bclass}$ is clearly open in $\gF$ and since they
are pairwise disjoint, each is closed as well.
Thus, each $\gF_{\Bclass}$ with
$\Bclass \in \Bset$, and hence also each
$\gF_{\Kclass} \subset \gF$ with $\Kclass \in \Kset$,
is a $\GF$-domain (i.e., a subset of $\gF$ that
is open, closed and invariant under $\GF$-conjugation).

Following \cite[Remark 2.2.2]{Kim04},
for $\Kclass \in \Kset \setminus \{*\}$, let
$\tlG_{\Kclass}$
(respectively, $\hatG_{\Kclass}$) be the set of
$\pi \in \tlG$ (respectively, $\pi \in \hatG$) such
that $\pi$ contains an element of $\Kclass$;
and then, following \cite[Definition 3.2.3]{Kim04},
for $\Bclass \in \Bset \setminus \{*'\}$, let
$\tlG_{\Bclass}$
(respectively, $\hatG_{\Bclass}$)
denote the set of all elements of $\tlG$ (respectively, $\hatG$)
containing an unrefined minimal $K$-type whose dual
blob belongs to $\Bclass$ (hence is good).
Let
$\tlG_* = \tlG_{*'}$ (respectively, $\hatG_* = \hatG_{*'}$)
denote the set of depth-zero representations
in $\tlG$ (respectively, $\hatG$).
Note that \(\tlG_{\Kclass}\) equals \(\tlG_{\Bclass}\)
and \(\hatG_{\Kclass}\) equals \(\hatG_{\Bclass}\)
whenever \(\Bclass = \iota(\Kclass)\).

\begin{remark} \label{rmk: GS GmcS}
As we recalled in the introduction,
\cite[Remark 2.2.2]{Kim04} defines a set $\tilde G_\oneK$
for every (not necessarily good) positive-depth unrefined minimal
$K$-type $\oneK$. Given such an $\oneK$,
we have by \cite[Lemma 1.4.5]{Kim04} that there exist
good, positive-depth unrefined minimal $K$-types that are weakly associate to $\oneK$
in the wider sense of \cite[Definition 2.2.1]{Kim04},
and by our prior discussion that the set of such good unrefined minimal \(K\)-types forms a weak associate class \(\Kclass\) in our sense.
By \cite[Theorem 4.5.1]{Kim04} (or \cite[Theorem 2.4.10]{KM03})
and the associativity of the unrefined minimal $K$-types contained in a fixed
element of $\tlG$ \cite[Theorem 3.5]{MP96}, we have that
$\tlG_\oneK$ equals the set $\tlG_{\Kclass}$ that we have
just defined. Moreover,
$\tlG$ and $\hatG$ are respectively the disjoint
unions of the sets $\tlG_{\Kclass}$ and the $\hatG_{\Kclass}$,
as $\Kclass$ runs over $\Kset$.
\end{remark}

\subsection{The Bernstein center and unrefined minimal $K$-types}

\begin{remark} \label{rmk: mcZ(G) review}
We briefly summarize some facts concerning the Bernstein center
$\mcZ(\G)$ of $\G$, referring the reader to \cite[Section 1.3]{BKV16}
or to \cite{Ber84} for more details. The ring $\mcZ(\G)$
can be realized
as the ring $\CC[\Omega(\G)]$ of regular functions on $\Omega(\G)$.
Giving an element $z \in \mcZ(\G)$ is equivalent to giving
an endomorphism $\pi(z)$ of each object $\pi$ in the category
of smooth representations of $\GF$,
such that \(\pi \mapsto \pi(z)\) is an endomorphism of the identity functor of that category (i.e., such that it respects the morphisms in the category).
One thinks of $\pi(z)$ as specifying
a $\GF$-equivariant action of $z$ on the space
of $\pi$.
Note that, when \(\pi\) is irreducible, the only such action is by scalar multiplication.
Recall that every \(\pi \in \tlG\) determines an infinitesimal character \(\inf(\pi) \in \Omega(\G)\).
We can match our two perspectives on the Bernstein center by requiring that, for every \(z \in \mcZ(\G)\) and \(\pi \in \tlG\), the action \(\pi(z)\) is by multiplication by \(z(\inf(\pi))\).
One can also
realize $\mcZ(\G)$ as the set of invariant distributions $z$ on $\GF$
such that $z * f \in C_c^{\infty}(\GF)$ for all $f \in C_c^{\infty}(\GF)$;
this connects to the earlier description by thinking of
$z * f$ as $\pi(z)(f)$, where $\pi = \ell$ is the left-regular representation
of $\GF$ on $C_c^{\infty}(\GF)$.
\end{remark}

\begin{lm} \label{lm: types are types}
For every
weak associate class $\Kclass$ of positive-depth, good unrefined
minimal $K$-types for $\G$,
the subset $\tlG_{\Kclass} \subset \tlG$
is a union of Bernstein components of $\G$.
\end{lm}

\begin{remark}
\label{rmk: GS types are types}
By Remark \ref{rmk: GS GmcS},
Lemma \ref{lm: types are types} is equivalent to the claim that,
for every positive-depth unrefined minimal $K$-type
$\oneK$ for $\G$, the subset $\tlG_\oneK \subset \tlG$
defined in the introduction
is a union of Bernstein components of $\G$.
It is this that we prove.
\end{remark}

\begin{proof}
Our justification will be
along the lines of the second proof of
\cite[Theorem 4.5.1(2)]{Kim04}.

We first claim that, for every positive-depth Bernstein component \(\mathfrak B\), there is a positive-depth unrefined minimal \(K\)-type \(\oneK_{\mathfrak B}\) that is contained in every representation in \(\mathfrak B\).
Indeed, consider the Bernstein component indexed by the inertial equivalence class of \((\M, \sigma)\), where \(\sigma\) has positive depth \(r\).
By \cite[Theorem 3.5(2)(ii)]{MP96}, there is a depth-\(r\), unrefined minimal \(K\)-type \(\oneK_\M\) contained in \(\sigma\).
Write \(\oneB_\M = Y + \mF^*_{x, {(-r)+}}\) for the dual blob of \(\oneK_\M\),
and put \(\oneK_{\mathfrak B} = (\GF_{x, r}, \Lambda_{Y, x, r})\).
If \(\pi\) lies in the chosen Bernstein component, then there is some unramified character \(\lambda_\M\) of \(\MF\) such that \(\pi\) has \((\M, \sigma \otimes \lambda_\M)\) as a supercuspidal support.  Note that \(\lambda_\M\) is trivial on every compact subgroup of \(\MF\), so that \(\sigma \otimes \lambda_\M\) also contains \(\oneK_\M\).
We have by \cite[Theorem 2.5]{MP96} that there is some parabolic subgroup \(\P\) of \(\G\) with Levi component \(\M\) such that \(\pi\) is a subrepresentation of the representation \(\Ind_\P^\G(\sigma)\) of $\GF$ obtained by unnormalized parabolic induction from $\sigma$. Hence by Frobenius reciprocity (e.g., in the form recalled at the end of \cite[Section 2.3]{MP96}), the Jacquet module of $\pi$ surjects to $\sigma$. Therefore, we have by \cite[Theorem 4.5]{MP96} that \(\pi\) contains a positive-depth, unrefined \(K\)-type \(\oneK_\pi = (\GF_{x, r}, \xi)\) that restricts to \(\oneK_\M\), and whose dual blob \(\oneB_\pi\) intersects \(\mF^*\).
Since, in this case, the dual blob \(\oneB_\M\) of the restriction \(\oneK_\M\) of \(\oneK_\pi\) is just \(\oneB_\pi \cap \mF^*\), it follows that \(\oneB_\pi\) contains \(\oneB_\M\), hence equals \(Y + \gF^*_{x, {(-r)+}}\).  That is, \(\oneK_\pi\) equals \((\GF_{x, r}, \Lambda_{Y, x, r}) = \oneK_{\mathfrak B}\); so we have shown that \(\pi\) contains \(\oneK_{\mathfrak B}\).

Now fix any positive-depth, unrefined minimal \(K\)-type \(\oneK\).
Remember that our notion of weak associativity
is not the same as the broader notion of \cite[Definition 2.2.1]{Kim04},
but rather is its restriction to good unrefined minimal \(K\)-types;
and that \(\tlG_\oneK\) is the set of representations containing a weak associate of \(\oneK\) in the broader sense.
Now suppose that \(\pi_1\) belongs to \(\tlG_\oneK\),
hence has positive depth (by \cite[Theorem 3.5]{MP96})
and contains some weak associate \(\oneK_1\) of \(\oneK\)
(in the broader sense);
and that \(\pi_2\) belongs to the same Bernstein component \(\mathfrak B\) as \(\pi_1\).
Then both \(\pi_1\) and \(\pi_2\) contain \(\oneK_{\mathfrak B}\).
By \cite[Theorem 3.5]{MP96} again, we have that \(\oneK_1\) and \(\oneK_{\mathfrak B}\) are associate.
This means that \(\oneK_{\mathfrak B}\), which is contained in \(\pi_2\),
is \emph{weakly} associate to \(\oneK\)
(in the broader sense),
hence that \(\pi_2\) belongs to \(\tlG_\oneK\).
By Remark \ref{rmk: GS types are types}, we have proven the result.
\end{proof}

\subsection{The Bernstein projector for a weak associate
class of good unrefined minimal $K$-types}
\label{subsec: projector}

Let $\Kclass$ be a weak associate class
of good (positive-depth) unrefined minimal $K$-types.
Write \(r\) for the depth of every \(K\)-type in \(\Kclass\).
As in the introduction, we let $E_{\Kclass}$ be the element
of $\mcZ(\G)$ that acts as the identity
on the set $\tlG_{\Kclass} \subset \tlG$
and is zero on all the elements of $\tlG \setminus \tlG_{\Kclass}$.

\begin{df} \label{df: S-projector}
Let $\Kclass$ be a weak associate class of
good unrefined minimal $K$-types of positive depth.
\begin{enumerate}
\item We call $E_{\Kclass}$
the $\Kclass$-projector or the Bernstein $\Kclass$-projector.
\item Let $\mcE_{\Kclass}$ be the distribution on $\gF$ obtained as the
inverse Fourier transform of the
distribution represented by the
characteristic function $\charfn_{\gF_{\Kclass}}$ of $\gF_{\Kclass}$
(which is a locally constant function,
as $\gF_{\Kclass}$ is open and closed in $\gF$).
\end{enumerate}
\end{df}

\begin{remark} \label{rmk: inverse Fourier transform}
Since we have imposed the Haar measure on $\gF$ that
is self-dual with respect to $\Lambda$ and
$B$ (see Remark \ref{rmk: hyp Kim Spi consequences}),
we have that $\mcE_{\Kclass}$ is also the Fourier transform
of the distribution on $\gF$ represented
by the function $X \mapsto \charfn_{\gF_{\Kclass}}(-X)$,
i.e.,
$\mcE_{\Kclass}$ is the Fourier transform of the distribution
represented by
the characteristic function $\charfn_{\gF_{\Kclass^{\vee}}}$ of
$\gF_{\Kclass^{\vee}}$.
\end{remark}

\section{The proof of Theorem \ref{thm: main}\ref{thm: log E} ---
the description of $E_{\Kclass}$ near the identity}

We now explain the sense in which Theorem \ref{thm: main}\ref{thm: log E} is a reformulation of
\cite[Theorem 3.3.1]{Kim04}.

\subsection{The Bernstein center and the Plancherel formula}

Recall that $\hatG \subset \tlG$ is the tempered dual of
$\G$. For every $f \in C_c^{\infty}(\GF)$, the Plancherel
formula gives an equation of the following form:
\begin{equation} \label{eqn: Plancherel formula}
f(1)
= \int_{\hatG} \Theta_{\pi}(f) \, d\pi
= \int_{\hatG} \Theta_{\pi^{\vee}}(f) \, d\pi,
\end{equation}
where $\pi^{\vee}$ denotes
the contragredient of $\pi$, and where
$d\pi$ refers to the Plancherel measure.
Here, we note that
\cite[p.~56, Equation (Pl)]{Kim04} gives the equality
of the first two terms, and that this automatically
gives the equality of the first term and the third.
Indeed, if $f^{\vee} \in C_c^{\infty}(\GF)$
is defined by requiring that $f^{\vee}(g) = f(g^{-1})$ for all $g \in \GF$,
then $f(1) = f^{\vee}(1)$, and
it is easy to see that for all $\pi \in \tlG$,
the operator $\pi^{\vee}(f)$ on the space of \(\pi^\vee\)
is the transpose of
the operator $\pi(f^{\vee})$ on the space of \(\pi\),
so that $\Theta_{\pi^{\vee}}(f)
= \Theta_{\pi}(f^{\vee})$.

With this notation, for all $z \in \mcZ(\G)$ and
$f \in C_c^{\infty}(\GF)$,
one can use the equality
of the first and second terms of Equation \eqref{eqn: Plancherel formula}
to write
in the spirit of \cite[Equation (2.4.3)]{MT02}:
\begin{equation} \label{eqn: 2.4.3 MT02}
z(f) = (z * f^{\vee})(1)
= \int_{\hatG} \Theta_{\pi}(z * f^{\vee}) \, d\pi
= \int_{\hatG} z(\pi) \Theta_{\pi}(f^{\vee}) \, d\pi
= \int_{\hatG} z(\pi) \Theta_{\pi^{\vee}}(f) \, d\pi,
\end{equation}
where in the left-most term, $z$ is thought of as
a distribution on $\GF$ and evaluated at
$f \in C_c^{\infty}(\GF)$,
while in the fourth and the fifth terms,
we have written $z(\pi)$ for the value $z(\inf(\pi))$ of $z \in
\mcZ(\G) = \CC[\Omega(\G)]$ at the image $\inf(\pi)$ of $\pi$ in
$\Omega(\G)$,
i.e., for the scalar by which \(\pi(z)\) acts on the space of \(\pi\);
or the equality of the first and the third
terms in Equation \eqref{eqn: Plancherel formula} to write
\begin{equation} \label{eqn: 2.4.3 MT02 second version}
z(f) = (z * f^{\vee})(1)
= \int_{\hatG} \Theta_{\pi^{\vee}}(z * f^{\vee}) \, d\pi
= \int_{\hatG} z(\pi^{\vee}) \Theta_{\pi^{\vee}}(f^{\vee}) \, d\pi
= \int_{\hatG} z(\pi^{\vee}) \Theta_{\pi}(f) \, d\pi.
\end{equation}

\subsection{Theorem \ref{thm: main}\ref{thm: log E}
as a reformulation of a result of Kim}

\begin{proof}[Proof of Theorem \ref{thm: main}\ref{thm: log E}]
For every $f \in C_c^{\infty}(\gF_r)$, we obtain
a function $f \circ \log \in C_c^{\infty}(\GF_r)$, which we view,
after extending by zero, as an element of $C_c^{\infty}(\GF)$.
Using \cite[Theorem 3.3.1]{Kim04},
we have for every $f \in C_c^{\infty}(\gF_r)$ that
\begin{equation} \label{eqn: Kim04 main input}
\int_{\gF_{\Kclass^{\vee}}} \hat f(X) \, dX
= \int_{\gF_{-\Bclass}} \hat f(X)\,dX
= \int_{\hatG_{-\Bclass}} \Theta_\pi(f \circ \log)\,d\pi
= \int_{\hatG_{\Kclass^{\vee}}} \Theta_{\pi}(f \circ \log) \, d\pi,
\end{equation}
where \(\Bclass = \iota(\Kclass)\),
provided we verify
that our choices of measures agree
with those of \cite{Kim04}.

The right-hand side of Equation \eqref{eqn: Kim04 main input}
is actually independent of the choice
of the measures once Equation \eqref{eqn: Plancherel formula}
is imposed, as is the case here as well as in \cite{Kim04}
(thus, this choice fixes
the product of the Haar measure on $\GF$ and the Plancherel measure).

On the left-hand side of Equation \eqref{eqn: Kim04 main input},
we have taken $dX$
to be the Haar measure on $\gF$
that is self-dual with respect to $\Lambda$ and $B$
(see Remark \ref{rmk: hyp Kim Spi consequences}).
The condition
``$\vol_{\gF}(\gF_{x, r})\vol_{\gF}(\gF_{x, (-r)+}) = 1$''
that appears below \cite[Equation (4)]{Kim04}
(and the fact that \(\gF_{x, r}\)
is the set of those elements \(Y \in \gF\) for which
\(\langle Y, \gF_{x, (-r)+} \rangle\)
		is contained in the maximal ideal of \(\mO\))
shows that \cite{Kim04} is also using this self-dual measure.

The left-hand side of Equation \eqref{eqn: Kim04 main input} equals
$\mcE_{\Kclass}(f)$ (see Remark \ref{rmk: inverse Fourier transform}).
Thus, it suffices to show that its right-hand side
equals $E_{\Kclass}(f \circ \log)$.
Since $\pi \in \hat \G_{\Kclass}$
if and only if $\pi^{\vee} \in \hat \G_{\Kclass^{\vee}}$,
this follows from Equation \eqref{eqn: 2.4.3 MT02 second version}.
\end{proof}

\begin{remark}
Thus, Theorem \ref{thm: main}\ref{thm: log E} is merely
a restatement of \cite[Theorem 3.3.1]{Kim04},
using the spectral description of elements of the Bernstein center
and the Plancherel formula.
In fact, the proof of Theorem \ref{thm: main}\ref{thm: E support} that
we give in Section \ref{sec: A proof of thm:main(i)}
amounts to adapting the strategy of
the proof of \cite[Theorem 3.3.1]{Kim04} to neighborhoods
of suitable non-identity elements of $\GF$,
via semisimple descent.
Just as the proof of \cite[Theorem 3.3.1]{Kim04} uses
the character expansion of \cite{KM03}, the proof
of Theorem \ref{thm: main}\ref{thm: log E} uses the second author's generalization
of the character expansion of \cite{KM03} to an
expansion about non-identity semisimple elements
of $\GF$, namely, \cite[Theorem 4.4.11]{Spi18}.
\end{remark}

\section{A proof of Theorem \ref{thm: main}\ref{thm: E support} ---
$E_{\Kclass}$ vanishes away from the identity}
\label{sec: A proof of thm:main(i)}
Let us begin with some preparation.
We keep the fixed weak associate class \(\Kclass\),
and its depth \(r\),
of Subsection \ref{subsec: projector}.
We now also fix as a good element $\Gamma \in - \gF_{\Kclass}
= \gF_{\Kclass^{\vee}}$,
and write $\G'$ for the centralizer of $\Gamma$ in $\G$.
Note that $\Gamma$ has depth $-r$
(see Subsubsection \ref{sssec: weak associativity remarks}),
and that
\((\GF_{x, r}, \Lambda_{-\Gamma, x, r})\) belongs to \(\Kclass\)
for every \(x \in \mcB(\G')\).

\subsection{Some functions on which $E_{\Kclass}$ acts
as the identity}

\begin{lm} \label{lm: non good types}
Let $x \in \mcB(\G') \subset \mcB(\G)$, let $X' \in
\gF'_{x, -r} \cap \gF'_{(-r)+}$, and consider
$\Lambda_{\Gamma + X', x, r} \in C_c^{\infty}(\GF_{x, r})
\subset C_c^{\infty}(\GF)$. Then
$E_{\Kclass} * \Lambda_{\Gamma + X', x, r} = \Lambda_{\Gamma + X', x, r}$.
\end{lm}
\begin{proof}
If $f_1, f_2 \in C_c^{\infty}(\GF)$, then, to show that $f_1 = f_2$,
it suffices by the Plancherel formula to show that
$\pi(f_1) = \pi(f_2)$ for each irreducible admissible representation
$\pi$ of $\GF$
(note that `$\pi(f_1) = \pi(f_2)$' refers to an equality of two linear
operators, and not just of their traces).
Thus, let $\pi \in \tlG$.
We need to show that $\pi(E_{\Kclass} * \Lambda_{\Gamma + X', x, r})
= \pi(\Lambda_{\Gamma + X', x, r})$.
Since $\pi(E_{\Kclass} * \Lambda_{\Gamma + X', x, r})
= \pi(E_{\Kclass}) \pi(\Lambda_{\Gamma + X', x, r})$,
it suffices to
show that if $\pi(\Lambda_{\Gamma + X', x, r}) \neq 0$,
then $\pi(E_{\Kclass})$ is the identity.
In other words, assuming that $\pi$ contains
$(\GF_{x, r}, \Lambda_{-\Gamma - X', x, r})$,
we need to show that $\pi$ contains an unrefined minimal
$K$-type which belongs to the weak associate class $\Kclass$.
Since $\Gamma$ is a good element
in $\gF_{\Kclass^{\vee}} = - \gF_{\Kclass}$,
any unrefined minimal $K$-type of the form
$(\GF_{y, r}, \Lambda_{-\Gamma, y, r})$, with $y \in \mcB(\G')$,
belongs to \(\Kclass\).
Thus, it suffices to show that, whenever
$\pi$ contains $(\GF_{x, r}, \Lambda_{- \Gamma - X', x, r})$,
it also contains $(\GF_{y, r}, \Lambda_{- \Gamma, y, r})$ for some
$y \in \mcB(\G')$. Since we are imposing Hypothesis \ref{hyp: Kim Spi}\ref{hyp: Kim 1.3.6},
which includes \cite[Hypothesis (HB)]{KM03},
the desired result follows from \cite[Lemma 2.4.11]{KM03}
(with our $-\Gamma$, $-X'$ and $r$ as
the $\Gamma$, $X'$ and $\varrho$, respectively, of that
lemma).
\end{proof}

\subsection{Review of asymptotic expansions around semisimple elements}
\label{subsec: asymptotic review}

In \cite[Theorem 5.3.1]{KM03}, Kim and Murnaghan proved that if $\pi$ is an
irreducible admissible representation of $\GF$ of
depth $r$, then
a variant of the Howe--Harish-Chandra
local character expansion for $\pi$ at the identity element
of $\GF$ is valid on $\gF_r$, which is bigger than the region
$\gF_{r+}$ on which the validity of the Howe--Harish-Chandra expansion
was proved by S.~DeBacker (following J.-L.~Waldspurger).
J.~Adler and J.~Korman used semisimple descent to generalize
the work of DeBacker
to give a range of validity for the
Howe--Harish-Chandra expansion around a non-identity semisimple element
\cite[Corollary 12.10]{AK07}.
Similarly, the second author generalized the work of \cite{KM03}
to give an asymptotic expansion
in the spirit of \cite{KM03} around many
semisimple elements of $\GF$, as well as an
explicit region for its validity that is in general bigger than
the analogous region in \cite{AK07}.

All these statements are
only true under appropriate technical hypotheses.
We will impose two hypotheses --- Hypotheses
\ref{hyp: Spi 1} and
\ref{hyp: Spi 2} --- that allow us to 
use some results, particularly Theorem 4.4.11, from \cite{Spi18}.
Of these, we now state
Hypothesis \ref{hyp: Spi 1} and proceed to
make a few observations towards setting the stage
for the statement of
Hypothesis \ref{hyp: Spi 2} in
Subsection \ref{subsec: hyp Spi}.

\begin{hyp} \label{hyp: Spi 1}
There is a collection $\{\gamma\}$
of semisimple elements of $\GF$,
containing \(\gamma = 1\),
with the following properties.
\begin{enumerate}[label=(\alph*)]
\item\label{hyp: good}
For each \(\gamma\) in the collection,
each eigenvalue $\lambda \in \bar F$ of $\Ad \gamma$ on
$\gF \otimes_F \bar F$ satisfies $\val(\lambda - 1) < r$
or \(\lambda = 1\).
\item\label{hyp: cover}
For each \(\gamma\) in the collection,
write \(\mathrm C_\G(\gamma)^\circ\) for the connected centralizer of \(\gamma\),
and \(\mcU_\gamma\) for the union \({^{\GF}(\gamma\cdot\mathrm C_\G(\gamma)^\circ(F)_r)}\)
of all $\GF$-conjugates of elements of $\gamma\cdot\mathrm C_\G(\gamma)^\circ(F)_r$.
Then \(G\) equals \(\bigcup_\gamma \mcU_\gamma\).
\end{enumerate}
\end{hyp}

Let $\gamma' \in \GF \setminus \GF_r$. We need
to show that the distribution $E_{\Kclass}$ is zero on a
neighborhood of $\gamma'$.
Hypothesis \ref{hyp: Spi 1} gives us
a semisimple element \(\gamma\) and
a $\GF$-conjugation invariant subset
$\mcU = \mcU_{\gamma} \subset \GF$
containing \(\gamma'\).
We will show in Remark \ref{rmk: U open} that $\mcU$ is open
and, once we have imposed Hypothesis \ref{hyp: Spi 2}, 
further in Subsection \ref{subsec: E vanishes} that
\(E_{\Kclass}\) vanishes on \(\mcU\),
hence around \(\gamma'\).
Set $\H = \mathrm C_\G(\gamma)$, and write $\H^{\circ}$
for the identity component of $\H$.
We slightly abuse notation by writing \(\HF^\circ\) for \(\H^\circ(F)\).

\begin{remark}
\label{rmk: gamma shallow}
We show that $\gamma$ does not belong to $\GF_r$.
Suppose it does, and let \(\T\) be a maximal torus of \(\G\) (necessarily split over a tame extension of \(F\), by Hypothesis \ref{hyp: Kim Spi}\ref{hyp: Kim 1.3.6}) containing \(\gamma\).
Then we have that \(\gamma\) (which is central in \(\H\)) lies in \(\TF\), hence in \(\H^{\circ}\), hence in the center \(\Z(\H^{\circ})\)
of $\H^{\circ}$.
By Corollary \ref{cor: H domain}, we have that \(\gamma\) lies in \(\HF_r\).
By \cite[Corollary 3.14]{AS08}, it follows that $\mcU = {^{\GF}(\gamma \HF_r)}$ is contained
in $\GF_r$, hence that \(\gamma' \in \mcU\) belongs to \(\GF_r\),
which is contrary to our assumption.
\end{remark}

In \cite[Theorem 4.4.11]{Spi18},
it is shown that, under some hypotheses, there is a Kim--Murnaghan-type
asymptotic expansion for the Harish-Chandra character
$\Theta_{\pi}$ of $\pi$ about $\gamma$
that is valid on $\mcU$.

Before reviewing this expansion, we make some informal remarks
to help the reader think of the relationship
between $\gamma'$ and $\gamma$.
These notions will be easier to relate to if one is
familiar with the notion of singular depth
from \cite[Definition 4.1]{AK07} and the relevance of this notion to the main
`range of validity' result of that paper,
\cite[Corollary 12.10]{AK07}. The reason
for considering $\gamma$ is that the singular
depth of the semisimple part of $\gamma'$ may be strictly bigger than $r$,
yielding a range of validity that is not
large enough for our purposes.
(Remember that \emph{larger} real numbers parameterize
\emph{smaller} sets in the Moy--Prasad filtration.)
To explain the relation
between $\gamma'$ and $\gamma$, recall that
in `good situations' we have a decomposition
$\gamma' = \gamma'_{< r} \gamma'_{\geq r}$ of
\(\gamma'\)
into a product of commuting elements (see \cite[Definition 6.8 ff.]{AS08}),
where
\(\gamma'_{< r}\) is semisimple with singular depth
less than $r$,
and $\gamma'_{\geq r}$ belongs to
$\mathrm C_\G(\gamma'_{< r})^\circ(F)_r$.
Then \(\gamma\) can be taken to be $\gamma'_{< r}$.

Let $T$ be an invariant distribution on
$\mcU = {^{\GF}(\gamma \HF_r)}$ (i.e., a linear
functional $C_c^{\infty}(\mcU) \to \CC$
that is invariant under precomposition with
$\GF$-conjugation).
Then the usual process of semisimple descent
\cite[Definition 7.3]{AK07}
gives an $\Ad \HF^{\circ}$-invariant distribution on
$\gamma \HF_r'$, in the notation of \cite[page 387]{AK07}.

\begin{lm}
\label{lm: submersive enough}
The set \(\HF_r'\) of \cite[page 387]{AK07} equals \(\HF_r\).
\end{lm}

\begin{proof}
Let us temporarily write
\(V = (\gF/\hF) \otimes_F \bar F\).
Recall that, by definition,
$\HF_r'$ is the set of $h \in \HF_r$ such that $\Ad(\gamma h) - 1$
is invertible on $\gF/\hF$, or, equivalently, on \(V\).
It suffices to show that, for every \(h \in \HF_r\),
no eigenvalue \(\mu\) of \(\Ad(h)\) on \(V\) equals an eigenvalue of \(\Ad(\gamma)\) there.
By Hypothesis \ref{hyp: Spi 1}\ref{hyp: good}, which says that the eigenvalues of \(\Ad(\gamma)\) on \(V\) all have valuation less than \(r\), it suffices to show that \(\mu - 1\) has valuation at least \(r\).
By \cite[Lemma 3.7.18]{AD02}, we have that the semisimple part \(h_{\text{ss}}\) of \(h\) belongs to \(\HF_r\).
Since the eigenvalues of \(\Ad(h)\) on \(V\)
are the same as those of its semisimple part \(\Ad(h_{\text{ss}})\),
we may, and do, assume that \(h\) is semisimple.
Since \(h\) lies in \(\HF_r \subset \HF^\circ\),
there is a maximal torus \(\T\) in \(\H\) containing \(h\),
so that Corollary \ref{cor: H domain} gives \(h \in \TF \cap \HF_r = \TF_r\).
In particular, for all absolute roots \(\alpha\) of \(\T\) in \(\Lie(\G)\) (not just \(\Lie(\H)\)), we have that \(\alpha(h) - 1\) has valuation at least \(r\).  Since each eigenvalue \(\mu\) of \(\Ad(h)\) is such a root value, we are done.
\end{proof}

\begin{remark}
\label{rmk: U open}
By the comment on submersivity just before
\cite[Theorem 7.1]{AK07}, and Lemma \ref{lm: submersive enough},
we have that
the map $\GF \times \HF_r \to
\GF$ given by $(g, h) \mapsto
g (\gamma h) g^{-1}$ is submersive
at each point,
and hence its
image, which is $\mcU = {^{\GF}(\gamma \HF_r)}$, is a
$\GF$-invariant open neighborhood of $\gamma$.
(Alternatively, we could use \cite[Lemma 3.2.11(iv)]{Spi18} to show that \(\mcU\) equals \(\bigcup_{x \in \mcB(\H)} {^{\GF}(\gamma\GF_{x, r})}\).)
Thus,
we have shown that $\mcU$ is open.  It is on this
neighborhood of $\gamma'$
that we will
show $E_{\Kclass}$ to vanish in Subsection \ref{subsec: E vanishes}.
\end{remark}

\begin{notn}
\label{notn: Tgamma}
Given any invariant distribution $T$ on $\mcU$, we
will denote by $T_{\gamma}$ its semisimple descent
to $\gamma \HF_r$
(see \cite[Definition 7.3]{AK07} and Lemma \ref{lm: submersive enough}).
\end{notn}

We remark that the process of semisimple
descent involves certain
choices of measures, which we make arbitrarily
and fix for the rest of this section. Up to a scalar whose value
is of no concern to us (and that could be absorbed into the choices of measures if desired), this is also the distribution denoted
$T_{\gamma}$ in \cite[Lemma 4.4.3]{Spi18}
(see the proof of that lemma and the comment just before it).

\begin{remark} \label{rmk: Tgamma determines T}
The distribution $T_{\gamma}$ determines $T$, thanks to
the surjectivity assertion of \cite[Theorem 7.1]{AK07}.
\end{remark}

Following the notation of \cite[page 2311]{Spi18},
write $\mcO^{\HF^{\circ}}({^{\GF}\Gamma})$ for the set of
$\HF^{\circ}$-orbits in $\hF$ whose closure
intersects the \(\GF\)-orbit ${^{\GF}\Gamma}$ of \(\Gamma\).
In other words, these
are the $\Ad \HF^{\circ}$-orbits of elements in $\hF$ whose
semisimple parts are $\GF$-conjugate to $\Gamma$.

\begin{remark}
\label{rmk: orbits finite}
We claim that $\mcO^{\HF^{\circ}}({^{\GF}\Gamma})$ is
finite.
Since \(F\) has characteristic \(0\), so that every connected, reductive group has finitely many rational orbits of nilpotent elements, it suffices to show that there are finitely many \(\HF^\circ\)-orbits in \({^{\GF}\Gamma} \cap \hF\).
This is \cite[Lemma 4.4.10]{Spi18},
but, to avoid assuming the hypotheses of \cite{Spi18},
we outline a proof in a similar spirit.
The proof simplifies because, by \cite[Chapter III, Example 4.2(d) and Theorem 4.4.5]{Ser97}, every \(\H^\circ(\bar F)\)-orbit in \(\Lie(\H)(\bar F)\) intersects \(\hF\) in finitely many \(\HF^\circ\)-orbits.

It thus suffices to show that there are
only finitely many $\H^{\circ}(\bar F)$-orbits in
$\Ad(\G(\bar F))(\Gamma) \cap \Lie(\H)(\bar F)$.
Fix a maximal torus \(\T\) in \(\H\).
Then every \(\H^\circ(\bar F)\)-orbit of semisimple elements in \(\Lie(\H)(\bar F)\) intersects \(\Lie(\T)(\bar F)\), so it suffices to show that
$\Ad(\G(\bar F))(\Gamma) \cap \Lie(\T)(\bar F)$ is
finite.
Since \(\T\) is also a maximal torus in \(\G\),
this last set is contained in the orbit of \(\Gamma\) under the Weyl group of \(\T\) in \(\G\), hence is finite.
\end{remark}

For each $\mcO \in \mcO^{\HF^\circ}({^{\GF}\Gamma})$, the choice of
an $\HF^{\circ}$-invariant measure on it gives us a distribution
$\nu_{\mcO}$ on $\hF$ supported on $\mcO$
\cite[Theorem 1]{Rao72},
and thus its Fourier transform $\hat \nu_{\mcO}$,
as in Subsection \ref{subsec: Fourier transform definition}.
Recall from Corollary \ref{cor: H domain} that $\log : \GF_r
\to \gF_r$ restricts to a homeomorphism $\HF_r \to \hF_r$
(which we also call $\log$).

\begin{df} \label{df: Gamma asymptotic expansion}
Let $T$ be an invariant distribution on $\mcU$.
Let $\theta_{T_{\gamma}}$ be the distribution on $\hF_r$ obtained
by pushing forward the semisimple descent $T_{\gamma}$ of $T$ from
$\gamma \HF_r$ to $\hF_r$ via $\gamma h \mapsto \log h$.
We say that $T$ has a $\Gamma$-asymptotic expansion
about \(\gamma\) if there exists
a tuple $(b_{\mcO})_{\mcO \in \mcO^{\HF^\circ}({^{\GF}\Gamma})}$
of complex numbers indexed by the elements of
$\mcO^{\HF^\circ}({^{\GF}\Gamma})$, such that
we have the following equality of distributions on
$\hF_r$:
\[
\theta_{T_{\gamma}} = \sum_{\mcO \in \mcO^{\HF^\circ}({^{\GF}\Gamma})}
b_{\mcO} \hat \nu_{\mcO} \res_{\hF_r}.
\]
In such a situation, the distribution
$\sum_{\mcO \in \mcO^{\HF^\circ}({^{\GF}\Gamma})}
b_{\mcO} \hat \nu_{\mcO}$ on $\hF$
will be referred to as a $\Gamma$-asymptotic expansion of
$T$ about $\gamma$ on $\mcU$.
\end{df}

\subsection{Hypotheses guaranteeing asymptotic expansions}
\label{subsec: hyp Spi}

Note that Hypothesis \ref{hyp: Spi 2} below involves
not only $\G$ but also our fixed weak associate class
$\Kclass$ (introduced in Subsection \ref{subsec: projector})
and the element $\Gamma \in \gF_{\Kclass^{\vee}}
= - \gF_{\Kclass}$ (introduced at the beginning of Section \ref{sec: A proof of thm:main(i)}).

\begin{hyp} \label{hyp: Spi 2}
Every element \(\gamma\) of the collection of semisimple
elements from Hypothesis \ref{hyp: Spi 1}
satisfies the following properties.
Let $\mcU = \mcU_{\gamma}$
be as in Hypothesis \ref{hyp: Spi 1}\ref{hyp: cover}.
\begin{enumerate}[label=(\alph*)]
\item\label{hyp: Spi thm 4.4.11} For every \(\pi\) in \(\tlG_{\Kclass}\),
the Harish-Chandra character $\Theta_{\pi^{\vee}}\res_{\mcU}$ has a $\Gamma$-asymptotic
character expansion about $\gamma$
in the sense of Definition \ref{df: Gamma asymptotic expansion}.

\item\label{hyp: Spi lem 4.4.14} Suppose that we have a distribution
\[ \theta = \sum_{\mcO \in \mcO^{\HF^\circ}({^{\GF}\Gamma})}
b_\mcO \hat \nu_{\mcO} \]
on $\hF$. Then $\theta = 0$ on \(\hF_r\)
if and only if
$\theta(\Lambda_{- X}\res_{\hF_{x, r}}) = 0$
for all $x$, $X \in \hF_{x, -r}$ and $g \in \GF$ such that
\begin{itemize}
\item $\Ad(g^{-1}) \Gamma \in \hF$;
\item $x \in \mcB(\H_g')$, where
$\H_g' \ceq \H \cap g^{-1} \G' g$; and
\item $X \in \Ad(g^{-1})(\Gamma + \gF'_{(-r)+})$.
\end{itemize}
Here, $\Lambda_{-X}\res_{\hF_{x, r}}$ is viewed as an element
of $C_c^{\infty}(\hF)$ that is zero outside $\hF_{x, r}$.
The element $X$ is denoted $X^*$ in \cite[page 2368]{Spi18}.
We have also used that $\H_g'$, being the centralizer in $\H$ of
the semisimple element $\Ad(g^{-1}) \Gamma$ of $\hF$, contains
a maximal torus of $\H$ and hence of $\G$,
which Hypothesis \ref{hyp: Kim Spi}\ref{hyp: Kim 1.3.6} shows is tame,
so that Subsection \ref{subsec: mcB embeddings} applies to
let us view $\mcB(\H_g')$ as a subset of $\mcB(\G)$.
\item\label{hyp: Spi lem 4.4.4} Whenever $x \in \mcB(\H)$ and
$X \in \hF_{x, -r}$,
we have that $\Lambda_{-X, x, r}\res_{\HF_{x, r}}$ and
$\Lambda_{-X, x, r}$ are `related by semisimple descent',
by which we mean that for all invariant distributions
$T$ on $\mcU$ with semisimple descent $T_{\gamma}$
onto $\gamma \HF_r$, we have:
\[
\vol_\GF(\GF_{x, r})^{-1}
T(\ell_{\gamma}(\Lambda_{-X, x, r}))
= \vol_\HF(\HF_{x, r})^{-1}
T_{\gamma}(\ell_{\gamma}(\Lambda_{-X, x, r}\res_{\HF_{x, r}})), \]
where \(\ell\) is the left regular representation.
We have written
\begin{itemize}
\item $\ell_{\gamma}(\Lambda_{-X, x, r})$ for the element
of $C_c^{\infty}(\GF)$ that is
supported on $\gamma \GF_{x, r}$, and given on
it by $\gamma g \mapsto \Lambda_{-X, x, r}(g)$;
and
\item $\ell_{\gamma}(\Lambda_{-X, x, r}\res_{\HF_{x, r}})$
for the element of $C_c^{\infty}(\HF)$ that is
supported on $\gamma \HF_{x, r}$, and given on it by
$\gamma h \mapsto \Lambda_{-X, x, r}(h)$.
\end{itemize}
\end{enumerate}
\end{hyp}

\begin{remark}
\label{rmk: Spi 2 in Spi18}
Hypothesis \ref{hyp: Spi 2} consists of statements that follow from \cite{Spi18}.  The reason that we do not cite the results of that paper directly is that we do not wish to recapitulate the lengthy list of hypotheses on which it depends.
(For the interested reader, these are
\cite[Hypotheses
	3.2.2,
	3.2.8,
	4.1.5,
	4.3.4,
	4.4.2,
	5.1.6%
]{Spi18}, all of which must hold for all the elements to which we apply them.)
\begin{itemize}
\item If \(\pi\) belongs to \(\tlG_{\Kclass}\), then \(\pi^\vee\) belongs to \(\tlG_{\Kclass^\vee}\), hence contains a good unrefined minimal \(K\)-type that belongs to \(\Kclass^\vee\).
Upon applying \cite[Lemma 1.4.2]{Kim04} to the dual blob of this \(K\)-type,
we see that $\pi^{\vee}$ contains a good unrefined
minimal $K$-type
whose dual blob $\Gamma' + \gF_{x', (-r)+}$
is contained in $\Gamma + \gF_{y, (-r)+}$
for some $y \in \mcB(\G')$. First, this gives that
$\gF_{x', (-r)+} \subset \gF_{y, (-r)+}$, hence, by taking duals,
that \(\gF_{x', r}\) is contained in \(\gF_{y, r}\);
so $\GF_{y, r} = \exp(\gF_{y, r})
\subset \exp(\gF_{x', r}) = \GF_{x', r}$. Moreover,
the containment
$\Gamma' + \gF_{x', (-r)+} \subset \Gamma + \gF_{y, (-r)+}$
of dual blobs also gives us that
$\Lambda_{\Gamma', x', r}\res_{\GF_{y, r}} = \Lambda_{\Gamma, y, r}$,
so that $\pi$ contains the good unrefined
minimal $K$-type $(\GF_{y, r}, \Lambda_{\Gamma, y, r})$.
We now see that Hypothesis \ref{hyp: Spi 2}\ref{hyp: Spi thm 4.4.11}
is a consequence of \cite[Theorem 4.4.11]{Spi18}
(with \(\Gamma\) in place of \(Z^*_o\)).
\item Hypothesis \ref{hyp: Spi 2}\ref{hyp: Spi lem 4.4.14} is a weaker statement than the one in \cite[Lemma 4.4.14]{Spi18}.
To see this, use that the
set $\mcO^{\HF^{\circ}}(\mcU^*)$ from that lemma is, by definition,
the set of $\HF^{\circ}$-orbits in
${^{\GF}(\Gamma + \gF'_{(-r)+})} \cap \hF$, and hence contains
the set $\mcO^{\HF^\circ}({^{\GF} \Gamma})$.
\item Hypothesis \ref{hyp: Spi 2}\ref{hyp: Spi lem 4.4.4} is a consequence of \cite[Lemma 4.4.4]{Spi18},
since the facts that \(x \in \mcB(\H)\) and \(X \in \hF\)
mean that the assignment $\phi \mapsto \hat \phi$
of that corollary takes $\Lambda_{X, x, r}\res_{\HF_{x, r}}$ on $\HF_{x, r}$ to $\Lambda_{X, x, r}$ on $\GF_{x, r}$.
\end{itemize}
\end{remark}

\subsection{$E_{\Kclass}$ has $\Gamma$-asymptotic
expansions}

Remember that we fixed an element \(\gamma' \in \GF \setminus \GF_r\) in Subsection \ref{subsec: asymptotic review}, around which we wanted to show that \(E_{\Kclass}\) vanished.  As preparation for this, we chose an element \(\gamma\) in the collection given by Hypothesis \ref{hyp: Spi 1}
(which also satisfies Hypothesis \ref{hyp: Spi 2})
such that the set \(\mcU = \mcU_\gamma\)
of Hypothesis \ref{hyp: Spi 1}\ref{hyp: cover}
---
which, by Remark \ref{rmk: U open}, is open
---
contains \(\gamma'\).
Remember that we put \(\H = \mathrm C_\G(\gamma)\).

\begin{lm} \label{lm: EmcS expansion}
$E_{\Kclass}\res_{\mcU}$ has a $\Gamma$-asymptotic
expansion about $\gamma$.
\end{lm}
\begin{remark}
A more `morally correct' way of proving this lemma might
be to mimick the proof of \cite[Theorem 4.4.11]{Spi18}.
However, the following proof, awkward as it may be,
has been chosen to spare the reader such
an effort.
\end{remark}
\begin{proof}[Proof of Lemma \ref{lm: EmcS expansion}]
For each $\pi \in \hatG_{\Kclass}$,
we have by Hypothesis \ref{hyp: Spi 2}\ref{hyp: Spi thm 4.4.11} that
$\Theta_{\pi^{\vee}}$ has a $\Gamma$-asymptotic
character expansion about $\gamma$ on $\mcU$.
Therefore it suffices to show that on $C_c^{\infty}(\mcU)$,
the distribution $E_{\Kclass}$, which, by Equation \eqref{eqn: 2.4.3 MT02},
is given by
\begin{equation} \label{eqn: EmcS}
E_{\Kclass}(f) = \int_{\hatG_{\Kclass}} \Theta_{\pi^{\vee}}(f) \, d\pi,
\end{equation}
is (noncanonically)
a finite complex linear combination of the
$\Theta_{\pi^{\vee}}\res_{C_c^{\infty}(\mcU)}$, as $\pi$ ranges
over $\hatG_{\Kclass}$.

Since $\mcO^{\HF^\circ}({^{\GF}\Gamma})$ is finite
(Remark \ref{rmk: orbits finite})
and a distribution is determined by its semisimple descent
(Remark \ref{rmk: Tgamma determines T}),
the space of distributions on $\mcU$ that have
a $\Gamma$-asymptotic expansion on $\mcU$ is finite dimensional.
Therefore, the complex linear span \(V\) of
$\set*
	{\Theta_{\pi^{\vee}}\res_{C_c^{\infty}(\mcU)}}
	{\pi \in \hatG_{\Kclass}}$
is a finite-dimensional vector space of distributions on $\mcU$.

Denote by $C_c^{\infty}(\mcU)_0$
the space of all $f \in C_c^\infty(\mcU)$ such that $\Theta_{\pi^{\vee}}(f) = 0$
for all $\pi \in \hatG_{\Kclass}$. Thus, for each $\pi \in \hatG_{\Kclass}$,
the restriction $\Theta_{\pi^{\vee}}\res_{C_c^{\infty}(\mcU)}$
lies in the subspace
$(C_c^{\infty}(\mcU)/C_c^{\infty}(\mcU)_0)^*$
of \(C_c^\infty(\mcU)^*\). So does
$E_{\Kclass}\res_{C_c^{\infty}(\mcU)}$, by Equation \eqref{eqn: EmcS}.
Thus,
it suffices to show that \(V\)
equals all of $(C_c^{\infty}(\mcU)/C_c^{\infty}(\mcU)_0)^*$.
Since we have shown that \(V\) is finite dimensional, it suffices to show
that it separates points on $C_c^{\infty}(\mcU)/C_c^{\infty}(\mcU)_0$.
But this follows from the definition of $C_c^{\infty}(\mcU)_0$.
\end{proof}

\subsection{Vanishing of $E_{\Kclass}$ on a domain}
\label{subsec: E vanishes}

Now we prove Theorem \ref{thm: main}\ref{thm: E support}.
We continue to assume Hypotheses \ref{hyp: Spi 1} and \ref{hyp: Spi 2}.
Recall that $\G'$ is the centralizer in \(\G\) of the element $\Gamma$
introduced at the beginning of Section \ref{sec: A proof of thm:main(i)}.

\begin{proof}[Proof of Theorem \ref{thm: main}\ref{thm: E support}]
In Subsection \ref{subsec: asymptotic review}, we picked out a semisimple element \(\gamma \in \GF \setminus \GF_r\), and showed that it would suffice to show that \(E_{\mathfrak s}\) vanishes on the set \(\mcU = \mcU_\gamma\) defined there.

By Lemma \ref{lm: EmcS expansion}, $E_{\Kclass}$ has
a $\Gamma$-asymptotic expansion on $\mcU$
(Definition \ref{df: Gamma asymptotic expansion}).
Namely, write \(E_{\Kclass, \gamma}\) for the semisimple descent of \(E_{\Kclass}\) to \(\gamma\HF_r\)
(Notation \ref{notn: Tgamma}),
and
$\theta$ for its push-forward to \(\hF_r\)
via $\gamma h \mapsto \gamma \log h$.
Then \(\theta\) is given by an expression
$\sum_{\mcO \in \mcO^{\HF^\circ}({^{\GF}\Gamma})} b_{\mcO} \hat \nu_{\mcO}$.
It is enough to show that $\theta = 0$ (on \(\hF_r\)).

Now we apply
Hypothesis \ref{hyp: Spi 2}\ref{hyp: Spi lem 4.4.14},
which says that, to check the desired equality, it is enough
to check that $\theta(\Lambda_{- X}\res_{\hF_{x, r}}) = 0$ for all
$x$, $X$ and $g$ as in that hypothesis.
We recall the notations \(\ell_\gamma(\Lambda_{-X, x, r}\res_{\HF_{x, r}})\) and \(\ell_\gamma(\Lambda_{-X, x, r})\) from Hypothesis \ref{hyp: Spi 2}\ref{hyp: Spi lem 4.4.4},
noting specifically that we regard them as functions on \(\HF\) and \(\hF\) by extension by \(0\).
By the definitions of \(\theta\) and $E_{\Kclass, \gamma}$,
we have
$\theta(\Lambda_{-X}\res_{\hF_{x, r}}) =
E_{\Kclass, \gamma}(\ell_{\gamma}(\Lambda_{-X, x, r}\res_{\HF_{x, r}}))$.

By using Hypothesis \ref{hyp: Spi 2}\ref{hyp: Spi lem 4.4.4}
(in the second of the two equalities below), we get:
\[
\vol_\HF(\HF_{x, r})^{-1}\theta(\Lambda_{-X}\res_{\hF_{x, r}})
= \vol_\HF(\HF_{x, r})^{-1}E_{\Kclass, \gamma}(\ell_{\gamma}(\Lambda_{-X, x, r}\res_{\HF_{x, r}}))
= \vol_\GF(\GF_{x, r})^{-1}E_{\Kclass}(\ell_{\gamma}(\Lambda_{-X, x, r})). \]

Thus, it is enough to show that $E_{\Kclass}(\ell_{\gamma}(\Lambda_{-X, x, r})) = 0$.
Now, recalling that $z(f) = (z * f^{\vee})(1)$ for
$z \in \mcZ(\G)$ and $f \in C_c^{\infty}(\GF)$, where
$f^{\vee} \in C_c^{\infty}(\GF)$ is given
by $g \mapsto f(g^{-1})$, we have:
\[
E_{\Kclass}(\ell_{\gamma}(\Lambda_{-X, x, r}))
= (E_{\Kclass} * \rho_{\gamma}(\Lambda_{X, x, r}))(1),
\]
where $\rho_{\gamma}(\Lambda_{X, x, r})$
is the element of $C_c^{\infty}(\GF)$ that is supported on
$\GF_{x, r} \gamma^{-1}$, and given on it by
$g \gamma^{-1} \mapsto \Lambda_{X, x, r}(g)$.
Since \(\gamma\) does not lie in \(\GF_{x, r}\), which is where \(\Lambda_{X, x, r}\) is supported, we have that \(\rho_\gamma(\Lambda_{X, x, r})\) vanishes at \(1\).  It thus remains to show that
$E_{\Kclass} * \rho_{\gamma}(\Lambda_{X, x, r})
= \rho_{\gamma}(\Lambda_{X, x, r})$.
For this, it suffices to show that $E_{\Kclass} * \Lambda_{X, x, r}
= \Lambda_{X, x, r}$
(because $f \mapsto z * f$ is $\GF \times \GF$-equivariant),
which in turn, by the same $\GF \times \GF$-equivariance,
follows if we show
that $E_{\Kclass} * \Lambda_{\Ad g(X), g \cdot x, r}
= \Lambda_{\Ad g(X), g \cdot x, r}$.
We claim that this last equality
follows from Lemma \ref{lm: non good types}.
To see this, note
(using Subsection \ref{subsec: mcB embeddings})
that, since \(x \in \mcB(\H'_g) \subseteq \mcB(g^{-1}\G'g)\)
we have that \(g\cdot x \in \mcB(\G')\),
and
\[ \Ad g(X) \in
\gF_{g \cdot x, -r} \cap (\Gamma + \gF'_{(-r)+})
= (\gF_{g \cdot x, -r} \cap \gF')
\cap (\Gamma + \gF'_{(-r)+})
= \gF'_{g \cdot x, -r} \cap (\Gamma + \gF'_{(-r)+}), \]
where the second equality follows from
Corollary \ref{cor: H domain}.
We have that $\Gamma$ is centralized by \(\GF'\), so belongs to the Lie algebra of every maximal torus in \(\G'\).
In particular, we may choose a maximal torus \(\T\) in \(\G'\) (even maximally split in \(\G'\)) such that \(g \cdot x\) belongs to \(\mcB(\T)\).
Since \(\Gamma\) is of depth \(-r\) in \(\G\), we have by Corollary \ref{cor: H domain} again that it belongs to \(\tF \cap \gF_{-r} = \tF_{-r} \subset \gF'_{g \cdot x, -r}\).
Thus we can apply Lemma \ref{lm: non good types} with
$g \cdot x$ being the $x$ of that lemma and with
$\Ad g(X) - \Gamma \in \gF'_{g \cdot x, -r} \cap \gF'_{(-r)+}$
being the \(X'\) of that lemma.
\end{proof}
\bibliographystyle{alpha}
\bibliography{bib}

\begin{thebibliography}{BKV16}

\bibitem[AD02]{AD02}
Jeffrey~D. Adler and Stephen DeBacker.
\newblock Some applications of {B}ruhat-{T}its theory to harmonic analysis on
  the {L}ie algebra of a reductive {$p$}-adic group.
\newblock {\em Michigan Math. J.}, 50(2):263--286, 2002.

\bibitem[AD04]{AD04}
Jeffrey~D. Adler and Stephen DeBacker.
\newblock Murnaghan-{K}irillov theory for supercuspidal representations of tame
  general linear groups.
\newblock {\em J. Reine Angew. Math.}, 575:1--35, 2004.

\bibitem[Adl98]{Adl98}
Jeffrey~D. Adler.
\newblock Refined anisotropic {$K$}-types and supercuspidal representations.
\newblock {\em Pacific J. Math.}, 185(1):1--32, 1998.

\bibitem[AK07]{AK07}
Jeffrey~D. Adler and Jonathan Korman.
\newblock The local character expansion near a tame, semisimple element.
\newblock {\em Amer. J. Math.}, 129(2):381--403, 2007.

\bibitem[AR00]{AR00}
Jeffrey~D. Adler and Alan Roche.
\newblock An intertwining result for {$p$}-adic groups.
\newblock {\em Canad. J. Math.}, 52(3):449--467, 2000.

\bibitem[AS08]{AS08}
Jeffrey~D. Adler and Loren Spice.
\newblock Good product expansions for tame elements of {$p$}-adic groups.
\newblock {\em Int. Math. Res. Pap. IMRP}, (1):Art. ID rp. 003, 95, 2008.

\bibitem[Ber84]{Ber84}
J.~N. Bernstein.
\newblock Le ``centre'' de {B}ernstein.
\newblock In {\em Representations of reductive groups over a local field},
  Travaux en Cours, pages 1--32. Hermann, Paris, 1984.
\newblock Edited by P. Deligne.

\bibitem[BKV16]{BKV16}
Roman Bezrukavnikov, David Kazhdan, and Yakov Varshavsky.
\newblock On the depth {$r$} {B}ernstein projector.
\newblock {\em Selecta Math. (N.S.)}, 22(4):2271--2311, 2016.

\bibitem[Bor91]{Bor98}
Armand Borel.
\newblock {\em Linear algebraic groups}, volume 126 of {\em Graduate Texts in
  Mathematics}.
\newblock Springer-Verlag, New York, second edition, 1991.

\bibitem[DR09]{DR09}
Stephen DeBacker and Mark Reeder.
\newblock Depth-zero supercuspidal {$L$}-packets and their stability.
\newblock {\em Ann. of Math. (2)}, 169(3):795--901, 2009.

\bibitem[Fin19]{Fin19}
Jessica Fintzen.
\newblock Tame tori in {$p$}-adic groups and good semisimple elements.
\newblock {\em Int. Math. Res. Not.}, 2019.

\bibitem[Kim99]{Kim99}
Ju-Lee Kim.
\newblock Hecke algebras of classical groups over {$p$}-adic fields and
  supercuspidal representations.
\newblock {\em Amer. J. Math.}, 121(5):967--1029, 1999.

\bibitem[Kim04]{Kim04}
Ju-Lee Kim.
\newblock Dual blobs and {P}lancherel formulas.
\newblock {\em Bull. Soc. Math. France}, 132(1):55--80, 2004.

\bibitem[Kim07]{Kim07}
Ju-Lee Kim.
\newblock Supercuspidal representations: an exhaustion theorem.
\newblock {\em J. Amer. Math. Soc.}, 20(2):273--320, 2007.

\bibitem[KM03]{KM03}
Ju-Lee Kim and Fiona Murnaghan.
\newblock Character expansions and unrefined minimal {$K$}-types.
\newblock {\em Amer. J. Math.}, 125(6):1199--1234, 2003.

\bibitem[KM06]{KM06}
Ju-Lee Kim and Fiona Murnaghan.
\newblock K-types and {$\Gamma$}-asymptotic expansions.
\newblock {\em J. Reine Angew. Math.}, 592:189--236, 2006.

\bibitem[MP94]{MP94}
Allen Moy and Gopal Prasad.
\newblock Unrefined minimal {$K$}-types for {$p$}-adic groups.
\newblock {\em Invent. Math.}, 116(1-3):393--408, 1994.

\bibitem[MP96]{MP96}
Allen Moy and Gopal Prasad.
\newblock Jacquet functors and unrefined minimal {$K$}-types.
\newblock {\em Comment. Math. Helv.}, 71(1):98--121, 1996.

\bibitem[MS20]{MS20}
Allen Moy and Gordan Savin.
\newblock Euler-poincar\'{e} formulae for positive depth bernstein projectors.
\newblock {\em arXiv preprint arXiv:1007.3576}, 2020.

\bibitem[MT02]{MT02}
Allen Moy and Marko Tadi\'{c}.
\newblock The {B}ernstein center in terms of invariant locally integrable
  functions.
\newblock {\em Represent. Theory}, 6:313--329, 2002.

\bibitem[RR72]{Rao72}
R.~Ranga~Rao.
\newblock Orbital integrals in reductive groups.
\newblock {\em Ann. of Math. (2)}, 96:505--510, 1972.

\bibitem[Ser97]{Ser97}
Jean-Pierre Serre.
\newblock {\em Galois cohomology}.
\newblock Springer-Verlag, Berlin, 1997.
\newblock Translated from the French by Patrick Ion and revised by the author.

\bibitem[Spi18]{Spi18}
Loren Spice.
\newblock Explicit asymptotic expansions for tame supercuspidal characters.
\newblock {\em Compos. Math.}, 154(11):2305--2378, 2018.

\bibitem[Yu01]{Yu01}
Jiu-Kang Yu.
\newblock Construction of tame supercuspidal representations.
\newblock {\em J. Amer. Math. Soc.}, 14(3):579--622, 2001.

\bibitem[Yu15]{Yu15}
Jiu-Kang Yu.
\newblock Smooth models associated to concave functions in {B}ruhat-{T}its
  theory.
\newblock In {\em Autour des sch\'{e}mas en groupes. {V}ol. {III}}, volume~47
  of {\em Panor. Synth\`eses}, pages 227--258. Soc. Math. France, Paris, 2015.

\end{thebibliography}

\end{document}